\documentclass[10pt,a4paper]{amsart}
\usepackage{amssymb,enumerate}
\usepackage{amsmath}
\usepackage{graphics}
\usepackage{mycjmunbm}
\usepackage{rotating}
\usepackage{url}
\newcommand{\F}{{\mathbb{F}}}
\newcommand{\C}{{\mathbb{C}}}
\newcommand{\Q}{{\mathbb{Q}}}
\newcommand{\Z}{{\mathbb{Z}}}
\newcommand{\ba}{{\boldsymbol{a}}}
\newcommand{\fS}{{\mathfrak{S}}}
\newcommand{\cB}{{\mathcal{B}}}
\newcommand{\cC}{{\mathcal{C}}}
\newcommand{\cF}{{\mathcal{F}}}
\newcommand{\cH}{{\mathcal{H}}}
\newcommand{\cI}{{\mathcal{I}}}
\newcommand{\cL}{{\mathcal{L}}}
\newcommand{\cN}{{\mathcal{N}}}
\newcommand{\cLR}{{\mathcal{LR}}}
\newcommand{\cO}{{\mathcal{O}}}
\newcommand{\cS}{{\mathcal{S}}}
\newcommand\Irr{\operatorname{Irr}}
\newcommand\Cl{\operatorname{Cl}}
\newcommand\Cmin{C_{\operatorname{min}}}
\newcommand\sgn{\operatorname{sgn}}
\renewcommand{\leq}{\leqslant}
\renewcommand{\geq}{\geqslant}
\renewcommand{\atop}[2]{\genfrac{}{}{0pt}{}{#1}{#2}}

\newtheorem{thm}{Theorem}[section]

\newtheorem{lem}[thm]{Lemma}
\newtheorem{cor}[thm]{Corollary}
\newtheorem{prop}[thm]{Proposition}
\newtheorem{defn}[thm]{Definition}
\newtheorem{exmp}[thm]{Example}
\newtheorem{rem}[thm]{Remark}
\newtheorem{conj}[thm]{Conjecture}

\begin{document}
\title{Some applications of CHEVIE to the theory of algebraic groups}
\author{Meinolf Geck}
\address{Institute of Mathematics \newline 
\indent University of Aberdeen  \newline
\indent Aberdeen AB24 3UE, UK}
\email{m.geck@abdn.ac.uk}

\setcounter{page}{1}
\coordinates{20}{2004}{1}{00-00}
\date{date}
\undate{date}

\subjclass[2000]{Primary 20C40, Secondary 20G40}
\keywords{\em Computer algebra, algebraic groups, Coxeter groups}

\begin{unabstract} 
The computer algebra system CHEVIE is designed to facilitate computations 
with various combinatorial structures arising in Lie theory, like finite 
Coxeter groups and Hecke algebras. We discuss some recent examples where 
CHEVIE has been helpful in the theory of algebraic groups, in 
questions related to unipotent classes, the Springer correspondence and 
Lusztig families. 
\end{unabstract} 

\maketitle
\pagestyle{myheadings}
\markboth{Geck}{Applications of CHEVIE to algebraic groups}

\section{Introduction} \label{secintro}
CHEVIE \cite{chevie} is a computer algebra project which was initiated 
about 20 years ago and has been further developed ever since; general 
information can be found on the webpage 
\begin{center}
\url{http://www.math.rwth-aachen.de/~CHEVIE}
\end{center}
which also contains links to various extensions and updates of CHEVIE.
The aim of CHEVIE is two-fold: firstly, it makes vast amounts of explicit 
data concerning Coxeter groups, Hecke algebras and groups of Lie 
type systematically available in electronic form; secondly, it provides 
tools, pre-defined functions and a programming environment (via its 
implementation in GAP \cite{gap} and MAPLE \cite{maple}) for performing 
symbolic calculations with these data. Through this combination, it has 
been helpful in a variety of applications; this help typically consists 
of:
\begin{itemize}
\item explicitly verifying certain properties (usually in the large groups 
of exceptional type) in the course of a case--by--case argument, or 
\item producing evidence in support of hypotheses and, conversely, 
searching for counter-examples, or 
\item performing experiments which may lead eventually to new theoretical 
insights (a conjecture, a theorem, a technique required in a proof, 
$\ldots$), 
\end{itemize}
or a combination of these. While the scope of CHEVIE is gradually expanding,
the original design has been particularly suited to algorithmic questions 
arising from Lusztig's work \cite{LuBook}, \cite{Lusztig03} on Hecke
algebras and characters of reductive groups over finite fields. 
 
The purpose of this article is to present selected examples of this 
interplay between theory and experimentation. The choice of examples is, 
of course, influenced by the author's own preferences. For quite some 
time now, algorithmic methods are well-established in various 
aspects of Lie theory (see, e.g., \cite{atlas}, \cite{fokko}, 
\cite{lie}), so another author---even another author from the CHEVIE 
project itself!---may easily come up with a completely different set of 
examples and applications.  

A finite Coxeter group $W$ can be described by a presentation with 
generators and defining relations, or by its action on a root system
in some Euclidean space. Thus, they are particulary suitable for 
the application of algorithmic methods. 
In Section~\ref{secconj}, we consider the conjugacy classes of $W$, 
especially questions related to elements of minimal length in the 
various classes---which is one of the areas where CHEVIE has been 
extremely helpful from its very beginnings; see \cite{gepf0}, 
\cite{gemi}. By recent work of Lusztig \cite{Lu10}, this plays a 
role in the construction of a remarkable map from conjugacy classes 
in a finite Weyl group to the unipotent classes in a corresponding 
algebraic group; this will be explained in Section~\ref{secbruhat}. 

In Section~\ref{secchar}, we shall consider certain standard operations 
in the character ring of $W$, like tensoring with the sign character and 
induction from parabolic subgroups---an area where one can use the 
full power of the highly efficient GAP functionality for character 
tables of finite groups. These operations are the combinatorial
counter-part of a number of constructions related to unipotent classes 
in algebraic groups and Lusztig's families of representations.

Finally, in Section~\ref{secgreen}, we consider the problem of computing
the Green functions of a finite group of Lie type. These functions provide
a substantial piece of information towards the determination of the whole
character table of such a group. The algorithm described by Shoji \cite{S1} 
and Lusztig \cite[\S 24]{L4} is now known to work without any restriction 
on the characteristic, and we explain how this can be turned into an 
efficient GAP program. A remarkable formula combining Green functions, 
character values of Hecke algebras and Fourier matrices is used in 
Lusztig's work \cite{Lu10} (mentioned above) to deal with groups of 
exceptional type---a highlight in the applications of CHEVIE.

While most of the content of these notes is drawn from existing sources,
there are a few items which are new; see, for example, the general 
existence result for excellent elements in the conjugacy classes of 
finite Coxeter groups in Section~\ref{secconj} and the 
characterisation of the $\ba$-function in Section~\ref{secchar}. 
We also mention our presentation of the algorithmic questions around
the computation of Green functions and Lusztig's results \cite{Lu10} 
in Section~\ref{secgreen}; in particular, we develop in somewhat more 
detail the fact that the $\F_q$-rational points in the intersections 
of Bruhat cells with unipotent classes can be counted by ``polynomials 
in~$q$''. This, and the experimental results in \cite{GeMa1}, lead 
us to conjecture the existence of a natural map from the conjugacy 
classes of $W$ to the Lusztig families of $W$; see Remark~\ref{finh}.

We assume that the reader has some familiarity with the general theory 
of (finite) Coxeter groups, the character theory of finite groups, and 
basic notions about algebraic groups; see, for example, \cite{Carter2}, 
\cite{gepf}, \cite{myag}. The manual of the GAP part of CHEVIE (available 
online in GAP or on the above webpage) may actually be a good place to 
start to read about the algorithmic theory of Coxeter groups. 

This is not meant to be a comprehensive survey about applications of 
CHEVIE. The interested reader may consult the bibliography for further 
reading; see, for example, Achar--Aubert \cite{aa}, Bellamy \cite{bell}, 
Casselman \cite{Cass}, Gomi \cite{gomi}, He \cite{he}, Himstedt--Huang 
\cite{him1}, Lusztig \cite{Lu03}, Reeder \cite{reed1}, to mention but a 
few from a variety of topics. Finally, Michel's development version 
\cite{jmich} of CHEVIE contains a wealth of material around complex 
reflection groups and ``{\em Spetses}'' \cite{spets}, a subject that we 
do not touch upon at all. 

\section{Conjugacy classes of finite Coxeter groups} \label{secconj}

Let $W$ be a finite Coxeter group, with generating set $S$ and 
corresponding length function $l\colon W \rightarrow \Z_{\geq 0}$. 
In CHEVIE, such a group is realised as a GAP permutation group via 
its action on the underlying root system; this provides highly 
efficient ways of performing computations with the elements of $W$ 
(multiplication, length function, reduced expressions, $\ldots$); 
see \cite[\S 2.2]{chevie}. 

We shall now explain some results on conjugacy classes which have been 
found and established through experiments with CHEVIE.

Let $\Cl(W)$ be the set of all conjugacy classes of $W$. For $C \in 
\Cl(W)$, let 
\[ d_C:=\min \{l(w) \mid w \in C\} \qquad \mbox{and} \qquad 
\Cmin :=\{w \in C \mid l(w)=d_C\}.\]
Thus, $\Cmin$ is the set of elements of minimal length in $C$. 
For any subset $I \subseteq S$, let $W_I\subseteq W$ be the 
parabolic subgroup generated by $I$. We say that $C \in \Cl(W)$ 
is {\em cuspidal} if $C \cap W_I=\varnothing$ for all proper 
subsets $I\subsetneqq S$. (These classes may also be called 
{\em anisotropic} or {\em elliptic}.) One can show that $C$ is cuspidal 
if and only if $\Cmin \cap W_I=\varnothing$ for all proper subsets 
$I\subsetneqq S$; see \cite[3.1.12]{gepf}. 

Let $w,w'\in W$. We write $w \rightarrow w'$ if there are sequences
of elements $w=y_0,y_1,\ldots,y_n=w'$ in $W$ and generators $s_1,\ldots,
s_n \in S$ such that, for each $i \in \{1,\ldots,n\}$, we have
$y_i=s_iy_{i-1}s_i$ and $l(y_i)\leq l(y_{i-1})$. This is a pre-order 
relation on $W$. Let $\leftrightarrow$ denote the associated equivalence 
relation, that is, we have $y \leftrightarrow w$ if and only if 
$y \rightarrow w$ and $w \rightarrow y$. The equivalence classes are 
called the {\em cyclic shift classes} of $W$; see \cite[3.2.3]{gepf}. 
Note that all elements in a fixed cyclic shift class have the same
length. Clearly, every conjugacy class of $W$ is a union of (several, 
in general) cyclic shift classes.

\begin{prop}[See \protect{\cite[3.2.7]{gepf}}] \label{prop1} Let 
$C\in \Cl(W)$ be cuspidal. Then the elements of $\Cmin$ form 
a single cyclic shift class.
\end{prop}

The proof of this result essentially relies on computer calculations, 
performed originally in \cite{gepf0}; see also \cite[\S 3.2]{chevie}, 
\cite[\S 3.3]{gepf}.

Using the concept of cuspidal classes, we obtain a full classification
of the conjugacy classes of $W$. To state the following result, let us
denote by $\cI(W,S)$ the set of all pairs $(I,C')$ where $I\subseteq
S$ and  $C'\in \Cl(W_I)$ is cuspidal (in $W_I$). Given two such pairs 
$(I_1,C_1')$ and $(I_2,C_2')$, we write $(I_1,C_1') \sim (I_2,C_2')$ if 
there exists some $x \in W$ such that $I_2=xI_1x^{-1}$ and $C_2'=xC_1'
x^{-1}$. 

\begin{thm}[Classification of $\Cl(W)$, \protect{\cite[3.2.12]{gepf}}]
\label{thm1} Let $C \in \Cl(W)$. Then the pairs $(I,C')$, where 
$I\subseteq S$ is the set of generators involved in a reduced 
expression of some $w\in \Cmin$ and $C'$ is the conjugacy class of $w$ 
in $W_I$, form an equivalence class in $\cI(W,S)$. Furthermore, we 
obtain a bijection
\[\Cl(W)\quad\stackrel{1{-}1}{\longrightarrow}\quad\cI(W,S)/\sim\]
by sending $C \in \Cl(W)$ to the equivalence class of pairs 
$(I,C')$ as above.
\end{thm}

(Again, the proof heavily relies on computer calculations.)

The above two results combined show that many properties about 
conjugacy classes of $W$ in general can be reduced to the study of 
suitable elements in cuspidal classes of $W$. Following recent work 
of Lusztig \cite{Lu10}, we will now discuss some special properties 
of the elements of minimal length in the classes of $W$. Let 
\[ T:=\{wsw^{-1} \mid w \in W, s \in S\}\]
be the set of reflections in $W$. 

\begin{lem} \label{lem1} Let $t \in T$. Then $t$ can be written in the
form $t=ysy^{-1}$ where $y \in W$ and $s \in S$ are such that $l(t)=2l(y)+1$.
\end{lem}

\begin{proof} Since $t$ has order $2$, we can apply the argument in the 
proof of \cite[3.2.10]{gepf}. This shows that there exists a subset 
$J \subseteq S$ and an element $y \in W$ such that $t=yw_Jy^{-1}$ where 
$w_J$ is the longest element in $W_J$; furthermore, $w_J$ is central in 
$W_J$ and $l(t)=2l(y)+l(w_J)$. It follows that $t$ has $|J|$ eigenvalues 
equal to $-1$ in the standard reflection representation of $W$. Since $t$ 
is a reflection, this forces that $|J|=1$. So we have $w_J=s$ for some
$s \in S$, as required.
\end{proof}

\begin{defn}[Lusztig \protect{\cite[2.1]{Lu10}}] \label{def1}  \rm
Let $C \in \Cl(W)$; suppose that $C$ corresponds to a pair $(I,C')$ as in
Theorem~\ref{thm1}. An element $w \in \Cmin$ is called {\em excellent}
if there exist reflections $t_1,\ldots,t_r \in T$, where $r=|I|$, such 
that 
\[w=t_1\cdots t_r\qquad\mbox{and}\qquad l(w)=l(t_1)+\cdots +l(t_r).\]
Thus, using Lemma~\ref{lem1}, an excellent element $w \in \Cmin$ admits
a reduced expression of the form 
\begin{align*}
w&=(s_1^1s_2^1 \cdots s_{q_1}^1s_{q_1+1}^1s_{q_1}^1\cdots s_2^1s_1)
(s_1^2s_2^2 \cdots s_{q_2}^2s_{q_2+1}^2s_{q_2}^2\cdots s_2^2s_2^2) \cdot\\
& \qquad \ldots \cdot (s_1^rs_2^r \cdots s_{q_r}^rs_{q_r+1}^rs_{q_r}^r
\cdots s_2^rs_1^r),
\end{align*}
where $s_i^j \in S$ for all $i,j$ and $l(w)=\sum_{1\leq j \leq r} 
(2q_j+1)$, as in \cite[2.1(a)]{Lu10}.
\end{defn}

Some examples are already mentioned in \cite[2.1]{Lu10}. In particular,
these show that, for a given class $C \in \Cl(W)$, there can exist
elements in $\Cmin$ which are not excellent. Lusztig also establishes 
the existence of excellent elements in all conjugacy classes of finite
Weyl groups, except when there is a component of type $E_7$ or $E_8$. Here 
we complete the picture by the following slightly stronger result, valid 
for all finite Coxeter groups.

\begin{prop} \label{prop2} Let $C \in \Cl(W)$; suppose that $C$ 
corresponds to a pair $(I,C')$ as in Theorem~\ref{thm1}. Then, for 
some element $w \in \Cmin$, there exist reflections $t_1,\ldots,t_r
\in T$, where $r=|I|$, with the following properties:
\begin{itemize}
\item[(a)] We have $w=t_1\cdots t_r$ and $l(w)=l(t_1)+\cdots +l(t_r)$;
thus, $w$ is excellent.
\item[(b)] There exist subsets $\varnothing =J_0 \subseteq J_1 
\subseteq \ldots \subseteq J_r \subseteq S$ such that, for $1\leq i
\leq r$, the reflection $t_i$ lies in $W_{J_i}$ and is a 
distinguished coset representative with respect to $W_{J_{i-1}}$, 
that is, we have $l(st_i)>l(t_i)$ for all $s \in J_{i-1}$.
\end{itemize}
\end{prop}

\begin{table}[htbp] \caption{Excellent elements in types $H_3$, 
$H_4$, $F_4$, $E_6$, $E_7$} \label{tab1}
\begin{center} {\footnotesize 
$\renewcommand{\arraystretch}{1.1} \begin{array}{cc} 
\begin{array}{c@{\hspace{1mm}}cl} \hline F_4 & d_C & 
\mbox{excellent $w\in \Cmin$}\\\hline
F_4 & 4 & (4)(3)(2)(1)\\
B_4 & 6 & (2)(4)(323)(1)\\
F_4(a_1) & 8 & (3)(4)(323)(121)\\
D_4 & 10 & (2)(323)(43234)(1)\\
C_3+A_1 & 10 & (1)(4)(3)(2132132)\\
D_4(a_1) & 12 & (3)(2)(43234)(12321)\\
A_3+\tilde{A}_1 & 14 & (2)(323)(43234)(12321)\\
A_2+\tilde{A}_2 & 16 & (2)(1)(4)(3213234321323)\\
4A_1 & 24 & (2)(323)(43234)\cdot \\
& & \qquad \cdot (123214321324321)\\
\hline & & \\ & & \end{array} & \quad 
\begin{array}{c@{\hspace{1mm}}cl} 
\hline E_6 & d_C & \mbox{excellent $w\in \Cmin$} \\\hline
E_6 & 6 & (1)(4)(2)(3)(6)(5)\\
E_6(a_1) & 8 & (1)(4)(3)(242)(5)(6)\\
E_6(a_2) & 12 & (3)(1)(5)(6)(34543)(242)\\
A_5+A_1 & 14 & (1)(2)(3)(6)(5)(423454234)\\
3A_2 & 24 & (1)(2)(3)(5)(6) \cdot \\
& & \quad \cdot (4315423456542314354)\\ \hline & & \\ 
\hline H_3 & d_C & \mbox{excellent $w\in \Cmin$} \\\hline
6 & 3 & (1)(2)(3)\\
8 & 5 & (1)(212)(3)\\
9 & 9 & (1)(212)(32123)\\
10 & 15 & (1)(3)(2121321213212)\\ \hline
\end{array} \end{array}$}
\end{center}
\begin{center} {\footnotesize 
$\begin{array}{c@{\hspace{1mm}}cl} \hline H_4 & d_C & 
\mbox{excellent $w\in \Cmin$} \\\hline
11 & 4 & (1)(2)(3)(4)\\
14 & 6 & (1)(212)(3)(4)\\
15 & 8 & (1)(2)(32123)(4)\\
17 & 10 & (1)(212)(32123)(4)\\
18 & 12 & (2)(1)(2123212)(343)\\
19 & 14 & (3)(2)(12132121321)(4)\\
21 & 16 & (1)(3)(2121321213212)(4)\\
22 & 16 & (1)(212)(32123)(4321234)\\
23 & 18 & (1)(212)(1321213)(4321234)\\
24 & 20 & (1)(2)(12132121321)(4321234)\\
25 & 22 & (1)(3)(2121321213212)(4321234)\\
26 & 24 & (1)(2)(4)(321213212343212132123)\\
27 & 26 & (2)(4)(121)(321213212343212132123)\\
28 & 28 & (1)(4)(212)(32121321432121321432123)\\
29 & 30 & (4)(3)(2)(123212132143212132124321213)\\
30 & 36 & (3)(2)(12132121321)(43212132123432121321234)\\
31 & 38 & (1)(3)(2121321213212)(43212132123432121321234)\\
32 & 40 & (1)(3)(4)(2132123432121321234321213212343212132)\\
33 & 48 & (1)(4)(212)(3212132123432121321234321213212343212132123)\\
34 & 60 & (1)(3)(2121321213212)\cdot \\
& & \qquad \cdot (432121321234321213212343212132123432121321234)\\ 
\hline & & \\ 
\hline E_7 & d_C & \mbox{excellent $w\in \Cmin$} \\\hline
E_7 & 7 & (7)(6)(5)(4)(3)(1)(2)\\
E_7(a_1) & 9 & (4)(7)(6)(5)(242)(3)(1)\\
E_7(a_2) & 11 & (5)(4)(7)(565)(242)(3)(1)\\
E_7(a_3) & 13 & (3)(5)(7)(6)(454)(23423)(1)\\
D_6+A_1 & 15 & (2)(3)(7)(6)(5)(423454234)(1)\\
A_7 & 17 & (2)(3)(6)(7)(565)(423454234)(1)\\
E_7(a_4) & 21 & (5)(6)(7)(45654)(2)(34543)(1234231)\\
D_6(a_2)+A_1 & 23 & (2)(3)(7)(6)(5)(423454234)(134565431)\\
A_5+A_2 & 25 & (3)(1)(2)(7)(6)(5)(4315423456542314354)\\
D_4+3A_1 & 31 & (2)(3)(5)(7)(423454234)(65423456765423456)(1)\\
2A_3+A_1 & 33 & (3)(1)(2)(5)(7)(423454234)(1654234567654231456)\\
7A_1 & 63 & (2)(3)(5)(7)(423454234) \cdot \\
& & \qquad \cdot (65423456765423456)(134254316542345676542314354265431)\\
\hline \end{array}$}
\end{center}
\end{table}

\begin{sidewaystable}[htbp]
\caption{Excellent elements in type $E_8$} \label{tab2}
\begin{center} {\footnotesize
$\renewcommand{\arraystretch}{1.06}\begin{array}{ccl}
\hline C & d_C & \mbox{excellent $w\in \Cmin$} \\\hline
E_8 &          8 & (1)(2)(3)(4)(7)(6)(5)(8)\\
E_8(a_1)    & 10 & (2)(4)(3)(1)(8)(7)(6)(454)\\
E_8(a_2)    & 12 & (5)(4)(7)(565)(343)(1)(2)(8)\\
E_8(a_4)    & 14 & (3)(1)(5)(343)(24542)(6)(7)(8)\\
E_8(a_5)    & 16 & (6)(5)(4)(3)(8)(2456542)(7)(131)\\
E_7+A_1     & 16 & (2)(3)(5)(423454234)(1)(8)(7)(6)\\
D_8         & 18 & (2)(3)(6)(5)(8)(676)(423454234)(1)\\
E_8(a_3)    & 20 & (2)(4)(3)(1)(423454234)(6)(8)(56765)\\
D_8(a_1)    & 22 & (4)(2)(3)(7)(6)(787)(5423456542345)(1)\\
E_8(a_7)    & 22 & (2)(5)(6)(454)(23423)(134565431)(7)(8)\\
E_8(a_6)    & 24 & (8)(7)(6)(5)(4)(2)(345676543)(123454231)\\
E_7(a_2)+A_1& 24 & (1)(2)(5)(6)(454)(314234565423143)(7)(8)\\
E_6+A_2     & 26 & (3)(1)(2)(5)(6)(8)(4315423456542314354)(7)\\
D_8(a_2)    & 26 & (2)(3)(5)(7)(6)(542345676542345)(8)(13431)\\
A_8         & 28 & (1)(2)(3)(8)(7)(6)(5)(431542345676542314354)\\
D_8(a_3)    & 30 & (1)(4)(2)(3)(7)(454)(316542345676542314356)(8)\\
D_6+2A_1    & 32 & (2)(3)(5)(8)(7)(423454234)(65423456765423456)(1)\\
A_7+A_1     & 34 & (3)(1)(2)(5)(7)(423454234)(1654234567654231456)(8)\\
E_8(a_8)    & 40 & (3)(4)(2)(131)(454)(234565423)(13456765431)
(24567876542)\\
E_7(a_4)+A_1& 42 & (2)(3)(4)(6)(131)(5423456542345)(1234567654231)
(456787654)\\
2D_4        & 44 & (2)(3)(5)(423454234)(1)(7)(65423456765423456)
(1345678765431)\\
E_6(a_2)+A_2& 44 & (3)(1)(2)(5)(6)(4315423456542314354)(23456765423)
(456787654)\\
A_5+A_2+A_1 & 46 & (2)(3)(5)(423454234)(1)(8)(7)
(6543176542345678765423143546576)\\
D_5(a_1)+A_3& 46 & (3)(1)(2)(5)(7)(6)(3425431654234567654231435426543)
(456787654)\\
2A_4        & 48 & (1)(2)(3)(5)(6)(7)(8)
(43542654317654234567876542314354265437654)\\
2D_4(a_1)   & 60 & (4)(2)(454)(3)(8)(7)(6542345678765423456)
(134254316542345676542314354265431)\\
D_4+4A_1    & 64 & (2)(3)(5)(7)(423454234)(65423456765423456)
(134254316542345676542314354265431)(8)\\
2A_3+2A_1   & 66 & (2)(3)(5)(7)(8)(423454234)(65423456765423456)
(13425431654234567876542314354265431)\\
4A_2        & 80 & (3)(1)(2)(6)(5)(8)(4315423456542314354)
(7654231435426543176542345678765423143542654317654234567)\\
8A_1 & 120 & (2)(3)(5)(7)(423454234)(65423456765423456)
(134254316542345676542314354265431)\cdot \\
& &\qquad\cdot (876542314354265431765423456787654231435426543176542345678)\\
\hline \end{array}$}\end{center}
\end{sidewaystable}

\begin{proof} By standard reduction arguments, we can assume that 
$(W,S)$ is irreducible. It will also be sufficient to deal with the 
case where $C$ is a cuspidal class. Now we consider the various
types of irreducible finite Coxeter groups.

First assume that $W$ is of type $I_2(m)$ where $m \geq 3$. Denote the 
two generators of $W$ by $s_1,s_2$. The cuspidal classes of $W$ are 
described in \cite[Exp.~3.2.8]{gepf}; representatives of minimal length 
are given by $w_i=(s_1s_2)^i$ where $1 \leq i \leq \lfloor m/2\rfloor$. 
We see that the decomposition $w_i=(s_1) (s_2s_1 \cdots s_1s_2)$ (where 
the second factor has length $2i-1$) satisfies the conditions (a) and (b).

If $W$ is of type $A_{n-1}$, then there is only one cuspidal
class $C$, namely, that containing the Coxeter elements. Furthermore,
$\Cmin$ consists precisely of the Coxeter elements; see 
\cite[3.1.16]{gepf}. Clearly, a reduced expression for a Coxeter 
element is a decomposition as a product of reflections which 
satisfies (a) and (b).

Next assume that $W$ is of type $B_n$ or $D_n$, where we use the 
following labelling of the generators of $W$:
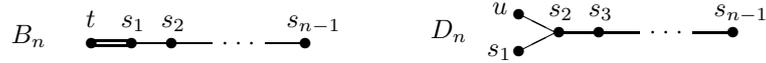
\begin{center}
\begin{picture}(280,18)
\put(  0, 5){$B_n$}
\put( 30, 5){\circle*{4}}
\put( 30, 4){\line(1,0){15}}
\put( 30, 6){\line(1,0){15}}
\put( 45, 5){\circle*{4}}
\put( 45, 5){\line(1,0){15}}
\put( 60, 5){\circle*{4}}
\put( 60, 5){\line(1,0){15}}
\put( 80, 5){\circle*{1}}
\put( 85, 5){\circle*{1}}
\put( 90, 5){\circle*{1}}
\put( 95, 5){\line(1,0){15}}
\put(110, 5){\circle*{4}}
\put( 28,11){$t$}
\put( 41,11){$s_1$}
\put( 56,11){$s_2$}
\put(103,11){$s_{n-1}$}
\put(157, 7){$D_n$}
\put(190, 2){\circle*{4}}
\put(190,16){\circle*{4}}
\put(190, 2){\line(2,1){15}}
\put(190,16){\line(2,-1){15}}
\put(205, 9){\circle*{4}}
\put(205, 9){\line(1,0){15}}
\put(220, 9){\circle*{4}}
\put(220, 9){\line(1,0){15}}
\put(240, 9){\circle*{1}}
\put(245, 9){\circle*{1}}
\put(250, 9){\circle*{1}}
\put(255, 9){\line(1,0){15}}
\put(270, 9){\circle*{4}}
\put(180,16){$u$}
\put(178, 0){$s_1$}
\put(201,15){$s_2$}
\put(216,15){$s_3$}
\put(263,15){$s_{n-1}$}
\end{picture}
\end{center}
The cuspidal classes of $W$ are parametrized by the partitions of
$n$ (with an even number of non-zero parts in type $D_n)$; see 
\cite[\S 2.2]{gemi} or \cite[\S 3.4]{gepf}. Let $C^\alpha\in\Cl(W)$ 
be the cuspidal class corresponding to the partition $\alpha$. A 
representative of minimal length in $C^\alpha$ is given as follows. For 
$1\leq i \leq n-1$, we set 
\[ \hat{s}_i:=\left\{ \begin{array}{cl} s_is_{i-1} \ldots  s_1ts_1 \ldots 
s_{i-1}s_i  &\quad \mbox{in type $B_n$,} \\ s_is_{i-1} \ldots s_2us_1s_2 
\ldots  s_{i-1}s_i &\quad \mbox{in type $D_n$.} \end{array} \right.\]
For $i=0$ we set $\hat{s}_0:=t$ (in type $B_n$) and $\hat{s}_0:=1$ (in type
$D_n$). Given $m \geq 0$ and $d \geq 1$, we define a ``negative
block'' of length $d$ and starting at $m$ by
\[ b^-(m,d):=\hat{s}_ms_{m+1}s_{m+2} \cdots s_{m+d-1}.\]
Now let $1\leq \alpha_1 \leq \alpha_2 \leq \ldots \leq \alpha_h$
be the non-zero parts of $\alpha$ (where $h$ is even if we are 
in type $D_n$). Let $m_i=\alpha_1+\cdots + \alpha_{i-1}$ for 
$i \geq 1$, where $m_1=0$. Then we have
\[w_\alpha:=b^-(m_1,\alpha_1)b^-(m_2,\alpha_2) \cdots 
b^-(m_h,\alpha_h) \in \Cmin^\alpha.\]
Note that $w_\alpha=t_1\cdots t_n$ where $t_1=\hat{s}_0$ and $t_i
\in \{s_{i-1}, \hat{s}_{i-1}\}$ for $i \geq 2$. 

Now, in type $B_n$, each $\hat{s}_i$ is a reflection. It easily 
follows that $w_\alpha$ is excellent (as already noticed by Lusztig 
\cite[2.2(a)]{Lu10}) and the additional requirements in (b) are
satisfied. The situation is slightly more complicated in type 
$D_n$, since $\hat{s}_i$ is not a reflection for $i \geq 1$. 
Lusztig \cite[2.3]{Lu10} already verified that $w_\alpha$ is
excellent but the expression for $w_\alpha$ as a product of
reflections described by Lusztig does not satisfy the conditions 
in (b). We need to somewhat modify $w_\alpha$ in order to make
sure that (b) holds. This is done as follows.
Since now $h$ is even, we can write 
\[ w_\alpha=(b_1b_2)(b_3b_4) \cdots (b_{h-1}b_{h}) \quad 
\mbox{where} \quad  b_i:=b^-(m_i,\alpha_i) \mbox{ for all $i$}.\]
By \cite[2.2]{gemi} (see also the proof of 
\cite[Lemma~2.6(b)]{gemi}), the factors $b_2,\ldots,b_{h}$ 
all commute with each other. On the other hand, note that $m_1=0$ 
and so $b_1=b^-(m_1,\alpha_1)=s_1s_2\cdots s_{\alpha_1-1}$. In this
case, we have $b_1b_i=b_i\tilde{b}_1$ and $\tilde{b}_1b_i=b_ib_1$ for 
any $i \geq 2$, where $\tilde{b}_1:=us_2 \ldots s_{\alpha_i-1}$. 
Since $h$ is even, this yields
\[ w_\alpha=b_1(b_3b_4) \cdots (b_{h-1}b_h)b_2=(b_{h-1}b_{h})
\cdots (b_3b_4)(b_1b_2).\]
Since every element in $W$ is conjugate to its inverse (see 
\cite[3.2.14]{gepf}), we obtain
\[ w_\alpha':=w_\alpha^{-1}=(b_1b_2)^{-1}(b_3b_4)^{-1}\cdots
(b_{h-1}b_{h})^{-1} \in \Cmin^\alpha.\]
Finally, we verify that each product $b_ib_{i+1}$ in the above expression
can be written in a suitable way as a product of reflections. First, we
compute:
\begin{align*}
b_1b_2&=(s_1s_2 \cdots s_{\alpha_1-1})(u_{\alpha_1}s_{\alpha_1+1}
\cdots s_{\alpha_1+\alpha_2-1})\\
& = (s_1\cdots s_{\alpha_1-1}s_{\alpha_1}s_{\alpha_1-1}\cdots s_1)
us_2\cdots s_{\alpha_1}s_{\alpha_1+1}\cdots s_{\alpha_1+\alpha_2-1}.
\end{align*}
Thus, we have $(b_1b_2)^{-1}=t_1\cdots t_{\alpha_1+\alpha_2}$ where 
\begin{align*}
t_1&=s_{\alpha_1+\alpha_2-1}, \quad t_2=s_{\alpha_1+\alpha_2-2},\quad
\ldots,\quad t_{\alpha_1+\alpha_2-2}=s_2,\quad t_{\alpha_1+\alpha_2-1}=
u,\\ t_{\alpha_1+\alpha_2}&=s_1\cdots s_{\alpha_1-1}s_{\alpha_1}
s_{\alpha_1-1}\cdots s_1;
\end{align*}
note that these are all reflections and $m_3=\alpha_1+\alpha_2$. 
Note also that the generators in $S$ which are involved in the
expression for $t_{\alpha_1+\alpha_2}$ are the ones which already
appeared in $t_1,\ldots,t_{\alpha_1+\alpha_2-1}$, together with $s_1$. 

Similarly, for $i \geq 3$, we find:
\begin{align*}
b_ib_{i+1} &= (u_{m_i}s_{m_i+1} \cdots s_{m_i+\alpha_i-1})(u_{m_i+
\alpha_i}s_{m_i+\alpha_i+1} \cdots s_{m_i+\alpha_i+\alpha_{i+1}-1})\\
&=(u_{m_i}s_{m_i+1} \cdots s_{m_i+\alpha_i-1}s_{m_i+\alpha_i}
s_{m_i+\alpha_i-1} \cdots s_{m_i+1}u_{m_i})\cdot\\
& \qquad\qquad\qquad \cdot s_{m_i+1}s_{m_i+2} 
\cdots s_{m_i+\alpha_i+\alpha_{i+1}-1}.
\end{align*}
Thus, we have $(b_ib_{i+1})^{-1}=t_{m_i+1} \cdots t_{m_i+\alpha_i+
\alpha_{i+1}}$ where
\begin{align*}
t_{m_i+1}&=s_{m_i+\alpha_i+\alpha_{i+1}-1}, \quad 
t_{m_i+2}= s_{m_i+\alpha_i+\alpha_{i+1}-2}, \\ & \quad \ldots, 
\quad t_{m_i+\alpha_i+ \alpha_{i+1}-1} = s_{m_i+1},\\
t_{m_i+\alpha_i+\alpha_{i+1}} &= u_{m_i}s_{m_i+1} \cdots 
s_{m_i+\alpha_i-1}s_{m_i+\alpha_i} s_{m_i+\alpha_i-1} \cdots 
s_{m_i+1}u_{m_i};
\end{align*}
note that these are all reflections and $m_{i+2}=m_i+\alpha_i+
\alpha_{i+1}$. Note also that the generators in $S$ which are involved 
in the expression for $t_{m_i+\alpha_i+\alpha_{i+1}}$ are the ones 
which already appeared in $t_1,\ldots,t_{m_i+\alpha_i+
\alpha_{i+1} -1}$, together with $s_{m_i}$. 

Combining these formulae, we obtain an expression $w_\alpha'=t_1
\cdots t_n$ such that condition (a) holds by construction. It is now 
also straightforward to verify that (b) holds. (This uses the 
above-mentioned information concerning the generators in $S$ which 
are involved in the expressions for the $t_i$; we omit further 
details.) Thus, the assertion is proved for $W$ of type $B_n$ and 
$D_n$.

Finally, in order to deal with the remaining groups of exceptional type, 
we use algorithmic methods and computer programs written in CHEVIE. This
involves the following steps. Let $C \in \Cl(W)$. An element $w \in 
\Cmin$ is explicitly specified in the tables in \cite[App.~B]{gepf}. 
First we compute the whole set $\Cmin$. By Proposition~\ref{prop1}, 
this set is the cyclic shift class containing $w$, and so it can be
effectively computed using Algorithm~G in \cite[\S 3.2]{gepf}. To procede,
it will be convenient to introduce the following notation. Given any 
element $w\in W$, we let $J(w)$ be the set of all $s\in S$ which appear 
in a reduced expression for $w$. (It is well-known that this does 
not depend on the choice of the reduced expression.) Then we say
that $w$ is {\em pre-excellent} if there exists a reflection 
$t \in T$ such that $l(wt)=l(w)-l(t)$ and $J(wt)\subsetneqq J(w)$.
These conditions can be effectively verified using the standard
programs available in CHEVIE. Given any subset $X \subseteq W$,
we define
\begin{align*}
X'&:=\{ w \in X \mid w \mbox{ pre-excellent}\},\\
\hat{X}&:=\{wt \mid w \in X', t \in T \mbox{ such 
that } l(wt)=l(w)-l(t) \mbox{ and } J(wt)\subsetneqq J(w)\}.
\end{align*}
Now we set $\cC_0:=\Cmin$ and then define recursively 
$\cC_i:=\hat{\cC}_{i-1}$ for $i=1,2,\ldots,|S|$. If the set 
$\hat{\cC}_{|S|}$ is non-empty and just contains the identity 
element then, clearly, the recursive procedure for reaching that 
set determines an element in $\Cmin$ together with a decomposition 
$w=t_1\cdots t_r$ as required in (a); furthermore, it yields subsets 
$\varnothing=J_0 \subsetneqq J_1\subsetneqq \ldots \subsetneqq J_r 
\subseteq W$ such that $t_i \in W_{J_i} \setminus W_{J_{i-1}}$ 
for $1 \leq i \leq r$. Given such a decomposition, it is then 
also straightforward to verify if the remaining conditions in 
(b) hold. 

It turns out that this procedure is successful for all $W$ of exceptional
type. The results are given in Tables~\ref{tab1} and \ref{tab2} (where we
use the notation of \cite[App.~B]{gepf}).
\end{proof}

We remark that condition (b) in Proposition~\ref{prop2} was essential in 
turning the question of the existence of excellent elements for the large
exceptional types into a feasible problem. In fact, the formulation of that 
condition itself was found by experiments with CHEVIE in small rank examples.

\section{Bruhat decomposition and unipotent classes} \label{secbruhat}

Following Lusztig \cite{Lu09}, \cite{Lu10}, the results and concepts 
discussed in the previous section can be seen to have a geometric 
significance. Let $k$ be an algebraic closure of the finite field $\F_p$ 
where $p$ is a prime. Let $G$ be a connected reductive algebraic group 
over $k$. Let $B \subseteq G$ be a Borel subgroup and $T \subseteq G$ be 
a maximal torus contained in $B$. Let $W=N_G(T)/T$ be the Weyl group of $G$, 
a finite Coxeter group. We have the Bruhat decomposition
\[ G=\coprod_{w \in W} B\dot{w}B \]
where $\dot{w}$ denotes a representative of $w \in W$ in $N_G(T)$.
Let $G_{\text{uni}}$ be the unipotent variety of $G$. It is known 
\cite{Lusztig76} that $G_{\text{uni}}$ is the union of finitely many 
conjugacy classes of $G$ which are called the {\em unipotent classes} 
of $G$. We can now state:

\begin{thm}[Lusztig \protect{\cite[0.4]{Lu10}}] \label{thm2} Assume
that $p$ is good for $G$. Let $C \in \Cl(W)$. Then there exists a unique 
unipotent class in $G$, denoted by $\cO_C$, with the following properties:
\begin{itemize}
\item[(a)] We have $\cO_C \cap B\dot{w}B \neq \varnothing$ for some 
$w \in \Cmin$.
\item[(b)] Given any $w'\in \Cmin$ and any unipotent class $\cO'$ 
we have $\cO' \cap B\dot{w}'B =\varnothing$, unless $\cO_C$ is contained
in the Zariski closure of $\cO'$.
\end{itemize}
Furthermore, the assignment $C \mapsto \cO_C$ defines a surjective map from
$\Cl(W)$ to the set of unipotent classes of $G$.
\end{thm}

Recall that $p$ is ``good'' for $G$ if $p$ is good for each simple factor 
involved in~$G$; the conditions for the various simple types are as follows.
\[\begin{array}{rl} A_n: & \mbox{no condition}, \\
B_n, C_n, D_n: & p \neq 2, \\
G_2, F_4, E_6, E_7: &  p \neq 2,3, \\
E_8: & p \neq 2,3,5.  \end{array}\]

\begin{rem} \label{rem1} \rm Let $C \in \Cl(W)$ and $\cO$ be a unipotent
class in $G$. Let $w,w' \in \Cmin$. As pointed out in \cite[0.2]{Lu10},
we have the equivalence:
\[ \cO \cap B\dot{w}B \neq \varnothing \quad \Leftrightarrow \qquad 
 \cO \cap B\dot{w}'B \neq \varnothing.\]
(This follows from Remark~\ref{rem2} and Corollary~\ref{cor1} below.) 
Hence, in condition (a) of the theorem we have in fact $\cO_C \cap 
B\dot{w}B \neq\varnothing$ for {\em all} $w \in \Cmin$.
\end{rem}

The {\em excellent} elements in the conjugacy classes of $W$ (see 
Definition~\ref{def1}) play a role in the proof of Theorem~\ref{thm2} for
$G$ of classical type. More generally, they enter the picture via the 
following conjecture which would provide an alternative 
and more direct description of the map $C \mapsto \cO_C$. 

\begin{conj}[Lusztig \protect{\cite[4.7]{Lu10}}] \label{conj1}
Let $C \in \Cl(W)$ and $w \in \Cmin$ be excellent, with a
decomposition $w=t_1\cdots t_r$ as in Definition~\ref{def1}. 
Define a corresponding unipotent element $u_w \in G$ as in 
\cite[2.4]{Lu10}. Then $u_w \in \cO_C$.
\end{conj}

\begin{exmp} \label{exp1} \rm Let $\Phi$ be the root system of $G$
with respect to $T$ and $\{\alpha_s \mid s \in S\} \subseteq 
\Phi$ be the system of simple roots determined by $B$. Let 
$X_\alpha=\{x_\alpha(\xi) \mid \xi \in k\} \subseteq G$ be the 
root subgroup corresponding to $\alpha \in \Phi$. Now let $s \in S$ 
and $C\in \Cl(W)$ be the conjugacy class containing $s$. Clearly, 
$s$ is excellent. By the procedure in \cite[2.4]{Lu10}, we obtain the 
unipotent element $u_s=x_{-\alpha_s}(1) \in G$; note that $u_s\in B 
\dot{s} B$. Then $\cO_C$ is the unipotent class containing $u_s$.
(This immediately follows from the reduction arguments in 
\cite[1.1]{Lu10}, which show that we can assume without loss
of generality that $W=\langle s\rangle$ and, hence, $G$ is a group of 
type $A_1$.)
\end{exmp}

\begin{rem} \label{rem2} \em Let $q$ be a power of $p$ and $F \colon G 
\rightarrow G$ be the Frobenius map with respect to a split 
$\F_q$-rational structure on $G$, such that $F(t)=t^q$ for all 
$t \in T$. Then $B$ and all unipotent classes of $G$ are $F$-stable; 
furthermore, $F$ acts as the identity on $W$. For each $w \in W$, we 
can choose $\dot{w}\in N_G(T)$ such that $F(\dot{w})=\dot{w}$. Given an
$F$-stable subset $M \subseteq G$, we write $M^F:=\{m \in M \mid 
F(m)=m\}$. Then, for any $w \in W$ and any unipotent class $\cO$ of 
$G$, we have the equivalence:
\begin{equation*}
\cO \cap B\dot{w}B \neq \varnothing \quad \Leftrightarrow \quad
|(\cO \cap B\dot{w}B)^F| \neq 0 \quad \mbox{for $q$ sufficiently large}.
\tag{a}
\end{equation*}
Hence, the conditions in Theorem~\ref{thm2} can be verified by
working in the finite groups $G^F$. (This remark already appeared
in \cite[1.2]{Lu10}.) 
\end{rem}

\begin{rem} \label{rem3} \rm The cardinalities on the right hand side 
of the equivalence in Remark~\ref{rem2} can be computed using the 
representation theory of the finite group $G^F$. Namely, consider 
the permutation module $\C[G^F/B^F]$ for $G^F$ and let
\[ \cH_q=\mbox{End}_{\C G^F}\bigl(\C[G^F/B^F]\bigr)^{\text{opp}}\]
be the corresponding {\em Hecke algebra}. (Here, ``opp'' denotes the 
opposite algebra; thus, $\cH_q$ acts on the right on $\C[G^F/B^F]$.) 
For $w \in W$, the linear map 
\[ T_w\colon \C[G^F/B^F]\rightarrow \C[G^F/B^F],\qquad xB^F \mapsto
\sum_{\atop{yB^F \in G^F/B^F}{x^{-1}y \in B^F\dot{w}B^F}} yB^F,\]
is contained in $\cH_q$. Furthermore, $\{T_w\mid w \in W\}$ is a 
basis of $\cH_q$ and the multiplication is given as follows, where
$s \in S$ and $w \in W$:
\[T_sT_w=\left\{\begin{array}{cl}T_{sw}&\qquad\mbox{if $l(sw)>l(w)$},
\\qT_{sw}+(q-1)T_w&\qquad\mbox{if $l(sw)<l(w)$};\end{array}\right.\]
see, for example, \cite[\S 67A]{CR2}, \cite[\S 8.4]{gepf}.  Now 
$\C[G^F/B^F]$ is a $(\C G^F,\cH_q)$-bimodule. For any $g\in G^F$ and 
$w\in W$, one easily finds using the defining formulae:
\[\mbox{trace}\bigl((g,T_w),\C[G^F/B^F]\bigr)=\frac{|C_{G^F}(g)|}{|B^F|}
\, |O_g \cap B^F\dot{w}B^F|\]
where $O_g$ denotes the conjugacy class of $g$ in $G^F$. Now, for any 
irreducible representation $V \in \Irr(\cH_q)$ there is a corresponding 
irreducible representation $\rho_V \in \Irr_\C(G^F)$, and this gives 
rise to a direct sum decomposition 
\begin{equation*}
\C[G^F/B^F]\cong\sum_{V \in \Irr(\cH_q)} \rho_V \otimes V\tag{a}
\end{equation*}
as $(\C G^F,\cH_q)$-bimodules; see, for example, \cite[\S 68B]{CR2},
\cite[8.4.4]{gepf}. In combination with the previous discussion, this 
yields the formula
\begin{equation*}
|O_g \cap B^F\dot{w}B^F|= \frac{|B^F|}{|C_{G^F}(g)|} \sum_{V \in 
\Irr(\cH_q)} \mbox{trace}(g,\rho_V) \,\mbox{trace}(T_w,V), \tag{b}
\end{equation*}
which already appeared in \cite[1.5(a)]{Lu09}. We shall illustrate 
the use of this formula in a small rank example below. Some more 
sophisticated techniques for the evaluation of the right hand side of
(b) will be discussed in Section~\ref{secgreen}.
\end{rem}

\begin{cor}[Lusztig \protect{\cite[1.5]{Lu09}, \cite[1.2]{Lu10}}]
\label{cor1} Let $\cO$ be a unipotent class in $G$. 
\begin{itemize}
\item[(a)] For a fixed $g \in \cO^F$, the linear map $\cH_q 
\rightarrow \C$, $T_w \mapsto |O_g \cap B^F\dot{w}B^F|$, 
is a trace function on $\cH_q$. 
\item[(b)] The linear map $\cH_q \rightarrow \C$, $T_w \mapsto 
|(\cO \cap B\dot{w}B)^F|$, is a trace function on $\cH_q$. 
\item[(c)] Let $C \in \Cl(W)$ and $w,w'\in \Cmin$. Then 
$|(\cO \cap B\dot{w}B)^F|=|(\cO \cap B\dot{w}'B)^F|$.
\end{itemize}
\end{cor}

\begin{proof} (a) The formula in Remark~\ref{rem3}(b) shows that the 
map $T_w \mapsto  |O_g \cap B^F\dot{w}B^F|$ is a $\C$-linear 
combination of characters of $\cH_q$ and, hence, a trace function. 

(b) First note that $(B\dot{w}B)^F=B^F\dot{w}B^F$. (This follows from 
the sharp form of the Bruhat decomposition; see \cite[2.5.13]{Carter2}, 
\cite[1.7.2]{myag}.) Now let $u_1,\ldots,u_d\in G^F$ be representatives 
of the $G^F$-conjugacy classes contained in $\cO^F$. Then 
\begin{align*}
|(\cO \cap B\dot{w}B)^F|&=|\cO^F \cap B^F\dot{w}B^F|=
\sum_{1\leq i \leq d} |O_{u_i} \cap B^F\dot{w}B^F|\\ &=
|B^F| \sum_{1\leq i \leq d} |C_{G^F}(u_i)|^{-1}
\mbox{trace}\bigl((u_i,T_w),\C[G^F/B^F]\bigr).
\end{align*}
So the assertion follows from (a).

(c) This is a general property of trace functions on $\cH_q$; 
see \cite[8.2.6]{gepf}.
\end{proof}

\begin{rem} \label{rem1a} \rm Lusztig's formulation \cite[0.4]{Lu10} of 
Theorem~\ref{thm2} looks somewhat different: Instead of using the 
intersections $\cO \cap B\dot{w}B$, he uses certain sub-varieties 
$\cB_w^\gamma\subseteq G \times G/B$ (where $\gamma$ denotes $\cO$). 
However, we have
\[ |(\cB_w^\gamma)^F|=\sum_{g \in \gamma^F} \mbox{trace}\bigl((g,T_w),
\C[G^F/B^F]\bigr)=|G^F/B^F|\,|(\cO \cap B\dot{w}B)^F|\]
where the first equality holds by \cite[1.2]{Lu10} and the second
by Remark~\ref{rem3} (see the proof of Corollary~\ref{cor1}(b)). In 
combination with Remark~\ref{rem2} we see that, indeed, the formulation 
of Theorem~\ref{thm2} is equivalent to Lusztig's version \cite{Lu10}.
\end{rem}

\begin{exmp} \label{exp2} \rm Let $G=\mbox{Sp}_4(k)$  where $W$ is 
of type $B_2$, with generators $S=\{s,t\}$. The algebra $\cH_q$ has
$5$ irreducible representations; their traces on basis elements
$T_w$ ($w \in \Cmin$) are given as follows; see \cite[Tab.~8.1, 
p.~270]{gepf}:
\[ \renewcommand{\arraystretch}{1.1} \begin{array}{cccccc} \hline  
 & T_1 & T_t & T_{stst} & T_{s} & T_{st} \\ \hline
\mbox{ind} & 1 & q & q^4 & q & q^2 \\
\sigma & 2 & q-1 & -2q^2 & q-1 & 0 \\ 
\sgn_1 & 1 & -1 & q^2 & q & -q \\
\sgn_2 & 1 & q & q^2 & -1 & -q\\
\sgn   & 1 & -1 & 1 & -1 & 1 \\
\hline \end{array}\]
Now assume that $\mbox{char}(k) \neq 2$. (Recall that $2$ is a bad prime
for type $B_2$.) There are four unipotent classes in $G$ which we denote
by $\cO_\mu$ where the subscript $\mu$ specifies the Jordan type of 
the elements in the class. For example, the class $\cO_{(211)}$
consists of unipotent matrices with one Jordan block of size $2$ and two 
blocks of size~$1$. The set $\cO_{(22)}^F$ splits into two classes 
in $G^F$ which we denote by $O_{(22)}$ and $O_{(22)}'$; each of the 
remaining classes $\cO_\mu$ gives rise to exactly one class in $G^F$ 
which we denote by $O_\mu$. The values of the irreducible characters 
of $G^F$ corresponding to $\Irr(\cH_q)$ can be extracted from 
Srinivasan's table \cite{Bhama}:
\[ \renewcommand{\arraystretch}{1.2} \begin{array}{cccccc} \hline  
& O_{(1111)} & O_{(211)} & O_{(22)} & O_{(22)}' & O_{(4)} \\ \hline 
|C_{G^F}(u)| & |G^F| & q^4(q^2-1) & 2q^3(q-1) &2q^3(q+1) & q^2 \\ \hline
\rho_{\text{ind}} & 1 & 1 & 1 & 1 & 1 \\
\rho_{\sigma}
&\frac{1}{2}q(q+1)^2 & \frac{1}{2}q(q+1) & q & 0 & 0 \\
\rho_{\sgn_1}
&\frac{1}{2}q(q^2+1) & -\frac{1}{2}q(q-1) & q & 0 & 0 \\
\rho_{\sgn_2}
&\frac{1}{2}q(q^2+1) & \frac{1}{2}q(q+1) & 0 & q & 0 \\
\rho_{\sgn} & q^4 & 0 & 0 & 0 & 0 \\ \hline 
\end{array}\]
We now multiply the transpose of the character table of $\cH_q$ 
with the above piece of Srinivasan's matrix. By the formula in 
Remark~\ref{rem3}(b), this yields (up to a factor $|C_{G^F}(u)|/
|B^F|$) the matrix of cardinalities $|O_u \cap B^F\dot{w} B^F|$ 
where $u \in G^F$ is unipotent and $w \in \Cmin$ for some $C \in 
\Cl(W)$: 
\[ \renewcommand{\arraystretch}{1.2} \begin{array}{cccccc} \hline  
& O_{(1111)} & O_{(211)} & O_{(22)} & O_{(22)}' & O_{(4)}\\ \hline 
1 & |G^F/B^F| & q^2 + 2q + 1 &  3q + 1 &  q + 1 & 1 \\
t & 0 & q^3 + q^2 & q^2 - q & q^2 + q & q \\
stst & 0 & 0 & q^4 - q^3 & q^4 + q^3 &  q^4 \\ 
s & 0 & 0 & 2q^2 & 0 & q \\
st & 0 & 0 & 0 & 0 & q^2 \\
\hline \end{array}\]
The closure relation among the unipotent classes is a linear order, 
in the sense that $\cO \subsetneqq \overline{\cO}'$ if and only
if $\dim \cO <\dim \cO'$. Thus, Theorem~\ref{thm2} yields the map
\[ C_1 \mapsto \cO_{(1111)}, \quad C_s \mapsto \cO_{(22)}, \quad 
C_t \mapsto \cO_{(211)}, \quad C_{st} \mapsto \cO_{(4)}, \quad 
C_{stst} \mapsto \cO_{(22)}\]
where $C_w$ denotes the conjugacy class of $W$ containing $w$.

Now assume that $\mbox{char}(k)=2$. We verify that, in this ``bad'' 
characteristic case, the assertions of Theorem~\ref{thm2} still hold. 
We use a similar convention for denoting unipotent classes as above; 
just note that, now, there are two unipotent classes in $G$ with 
elements of Jordan type $(22)$, which we denote by $\cO_{(22)}$ and 
$\cO_{(22)}^*$. The values of the irreducible characters of $G^F$ 
corresponding to $\Irr(\cH_q)$ have been determined by Enomoto 
\cite{enom1} (with some corrections due to L\"ubeck):
\[\renewcommand{\arraystretch}{1.2} \begin{array}{ccccccc} \hline  
& O_{(1111)} & O_{(211)} & O_{(22)}^* & O_{(22)} & O_{(4)} & O_{(4)}' 
\\ \hline |C_{G^F}(u)| 
& |G^F| & q^4(q^2-1) & q^4(q^2-1) & q^4 & 2q^2 & 2q^2 \\ \hline
\rho_{\text{ind}} & 1 & 1 & 1 & 1 & 1 & 1 \\
\rho_{\sigma}
&\frac{1}{2}q(q+1)^2 & \frac{1}{2}q(q+1) & \frac{1}{2}q(q+1)
& \frac{q}{2} & \frac{q}{2} &-\frac{q}{2} \\ 
\rho_{\sgn_{1}}
&\frac{1}{2}q(q^2+1) & -\frac{1}{2}q(q-1) & \frac{1}{2}q(q+1)
& \frac{q}{2} & -\frac{q}{2} &\frac{q}{2} \\ 
\rho_{\sgn_2}
&\frac{1}{2}q(q^2+1)^2 & \frac{1}{2}q(q+1) & -\frac{1}{2}q(q-1)
& \frac{q}{2} & -\frac{q}{2} &\frac{q}{2} \\ 
\rho_{\sgn} & q^4 & 0 & 0 & 0 & 0 & 0 \\ \hline \end{array}\]
As before, this yields (up to a factor $|C_{G^F}(u)|/|B^F|$) the 
matrix of cardinalities $|O_u \cap B^F\dot{w} B^F|$ where 
$u \in G^F$ is unipotent and $w \in \Cmin$ for some $C \in \Cl(W)$: 
\[ \renewcommand{\arraystretch}{1.2} \begin{array}{ccccccc} \hline  
& O_{(1111)} & O_{(211)} & O_{(22)}^* & O_{(22)} & O_{(4)} & O_{(4)}' 
\\ \hline 1 & |G^F/B^F| & q^2 + 2q + 1 & 
  q^2 + 2q + 1 & 2q + 1 & 1 & 1 \\ 
t& 0 & q^3 + q^2 & 0 & q^2 & q & q \\ 
stst& 0 & 0 & 0 & q^4 & q^4 - 2q^3 & q^4 + 2q^3 \\ 
s & 0 & 0 & q^3 + q^2 & q^2 & q & q \\ 
st& 0 & 0 & 0 & 0 & 2q^2 & 0 \\ 
\hline \end{array}\]
We conclude that the conditions in Theorem~\ref{thm2} hold for the map
\[ C_1 \mapsto \cO_{(1111)},\quad C_s \mapsto \cO_{(22)}^*,\quad 
C_t \mapsto \cO_{(211)}, \quad C_{st}  \mapsto \cO_{(4)},\quad 
C_{stst} \mapsto \cO_{(22)}.\]
As pointed out by Lusztig \cite[4.8]{Lu11}, there is considerable
evidence that, in general, Theorem~\ref{thm2} will continue to hold in 
bad characteristic. 
\end{exmp}

\section{Characters of finite Coxeter groups} \label{secchar}
All the general GAP functionality for working with character tables of 
finite groups is available for finite Coxeter groups: For example, we 
can form tensor products of characters, induce characters from subgroups,
and decompose the characters so obtained into irreducibles. For a finite
Coxeter group $W$, the following versions of the above operations are 
particularly relevant:
\begin{itemize}
\item tensoring with the sign character (usually denoted here by 
``$\sgn$'');
\item inducing characters from parabolic subgroups (or
reflection subgroups).
\end{itemize}
Beginning with \cite{Lusztig79b}, Lusztig developed the idea that 
various data which are important in the representation theory of 
reductive algebraic groups can be recovered purely in terms of the 
above operations together with certain numerical functions on the 
irreducible characters of $W$. (See Lusztig \cite{Lu09a} for more 
recent work in this direction.) Quite often this leads to explicit 
recursive descriptions of these data, which can be effectively 
implemented in programs written in the GAP language. We discuss some
examples in this section. 

Probably the most subtle of the numerical functions on the irreducible
characters of $W$ is given by the so-called ``$\ba$-invariants''. These 
are originally defined in \cite{Lusztig79b} by using the ``generic 
degrees'' of the corresponding generic Iwahori--Hecke algebra; see 
\cite[\S 68C]{CR2}, \cite[9.3.6]{gepf}. Developing an idea in 
\cite[\S 6.5]{gepf}, we begin by showing that these 
``$\ba$-invariants'' can be characterised purely in terms of the 
characters of $W$, without reference to the generic Iwahori--Hecke 
algebra. 

We shall work in the general ``multi-parameter'' setting of \cite{Lu83}. 
To describe this, let $\Gamma$ be an abelian group (written additively). 
Following Lusztig \cite{Lusztig03}, we say that a function $L \colon 
W\rightarrow\Gamma$ is a {\em weight function} if we have 
\[L(ww')= L(w)+L(w') \quad \mbox{for all $w,w'\in W$ such that 
$l(ww')=l(w)+l(w')$}.\]
Note that such a function $L$ is uniquely determined by the values 
$\{L(s)\mid s \in S\}$. Furthermore, if $\{c_s \mid s \in S\}$ is a 
collection of elements in $\Gamma$ such that $c_s=c_t$ whenever $s,t 
\in S$ are conjugate in $W$, then there is (unique) weight function 
$L\colon W \rightarrow \Gamma$ such that $L(s)=c_s$ for all $s \in S$. 
(This follows from Matsumoto's Lemma; see \cite[\S 1.2]{gepf}.) We will 
further assume that $\Gamma$ admits a total ordering $\leq$ which is 
compatible with the group structure, that is, whenever $g, g'\in\Gamma$ 
are such that $g \leq g'$, we have $g+h\leq g'+h$ for all $h \in 
\Gamma$. Then we will require that 
\[ L(s)\geq 0 \qquad \mbox{for all $s \in S$}.\]
(The standard and most important example of this whole setting is 
$\Gamma=\Z$ with its natural ordering; if, moreover, we have $L(s)=1$ for 
all $s \in S$, then we say that we are in the ``equal parameter case''.)

Let $\Irr(W)$ be the set of (complex) irreducible representations of $W$ 
(up to isomorphism). Having fixed $L, \Gamma,\leq$ as above, we wish to 
define a function
\[ \Irr(W) \rightarrow \Gamma_{\geq 0}, \qquad E \mapsto \tilde{\ba}_E.\] 
We need one further piece of notation. Recall that $T=\{wsw^{-1}\mid w 
\in W,s\in S\}$ is the set of all reflections in $W$. Let $S'\subseteq S$ 
be a set of representatives of the conjugacy classes of $W$ which are 
contained in $T$. For $s \in S'$, let $N_s$ be the cardinality of the 
conjugacy class of $s$; thus, $|T|= \sum_{s\in S'} N_s$. Now let $E \in 
\Irr(W)$ and $s\in S'$. Since $s$ has order $2$, it is clear that
$\mbox{trace}(s,E)\in \Z$. Hence, by a well-known result in the character 
theory of finite groups, the quantity $N_s\mbox{trace}(s,E)/\dim E$ is an 
integer. Thus, we can define
\[ \omega_L(E):=\sum_{s \in S'} \frac{N_s\, \mbox{trace}(s,E)}{\dim E}\,
L(s) \in \Gamma.\]
(Note that this does not depend on the choice of the set of representatives
$S' \subseteq S$.) 

\begin{defn} \label{def2} \rm  We define a function $\Irr(W) \rightarrow 
\Gamma$, $E \mapsto \tilde{\ba}_E$, inductively as follows. If $W=\{1\}$, 
then $\Irr(W)$ only consists of the unit representation (denoted $1_W$) 
and we set $\tilde{\ba}_{1_W}:=0$.  Now assume that $W \neq \{1\}$ and 
that the function $E \mapsto \tilde{\ba}_E$ has already been defined for 
all proper parabolic subgroups of $W$. Then, for any $E \in \Irr(W)$, we 
can define 
\[ \tilde{\ba}_E^\prime:= \max\{\tilde{\ba}_M \mid  M \in \Irr(W_J) 
\mbox{ where } J \subsetneqq S \mbox{ and } M \uparrow E\}.\]
Here, we write $M \uparrow E$ if $E$ is an irreducible constituent of the 
representation obtained by inducing $M$ from $W_J$ to $W$. Finally, we set 
\[ \renewcommand{\arraystretch}{1.2} \tilde{\ba}_E:=
\left\{ \begin{array}{cl} \tilde{\ba}_E^\prime & \quad 
\mbox{if $\tilde{\ba}_{E\otimes \sgn}^\prime-\tilde{\ba}_E^\prime 
\leq \omega_L(E)$},\\ \tilde{\ba}_{E \otimes \sgn}^\prime-\omega_L(E) & 
\quad \mbox{otherwise}.  \end{array}\right.\] 
One immediately checks that this function satisfies the following 
conditions:
\[ \tilde{\ba}_E \geq \tilde{\ba}_E^\prime \geq 0 \quad \mbox{and} 
\quad \tilde{\ba}_{E \otimes \sgn}-\tilde{\ba}_E=\omega_L(E) \quad 
\mbox{for all $E \in \Irr(W)$}.\]
This also shows that $\tilde{\ba}_E \geq \tilde{\ba}_M$ if
$M \uparrow E$ where $M \in \Irr(W_J)$ and $J \subsetneqq S$.
\end{defn}

\begin{exmp} \label{asym} \rm (a) If $L(s)=0$  for all $s \in S$, then 
$\tilde{\ba}_E=0$ for any $E \in \Irr(W)$. 

(b) Assume that we are in type $A_{n-1}$, where $W \cong \fS_n$ and there 
is a natural labelling $\Irr(W)=\{E^\lambda \mid \lambda \vdash n\}$. All 
generators in $S$ are conjugate and so any non-zero weight function $L$ 
takes a constant value $a>0$ on $S$. Then we have:
\[ \tilde{\ba}_{E^\lambda}=\sum_{1\leq i \leq r} (i-1)\lambda_i\, a
\quad \mbox{where} \quad \lambda=(\lambda_1\geq\lambda_2\geq\ldots 
\geq \lambda_r\geq 0);\]
This can be shown by a direct argument, as indicated in 
\cite[Example~6.5.8]{gepf}.
\end{exmp}

\begin{rem} \label{rem4} \rm As already mentioned, Lusztig originally
defined an ``$\ba$-function'' $\Irr(W) \rightarrow \Gamma$, $E \mapsto 
\ba_E$, using the ``generic degrees'' of the generic Iwahori--Hecke algebra
associated with $W$ and the weight function $L$. It is known that this 
function has the following properties:
\begin{itemize}
\item[(A0)] {\em We have $\ba_{1_W}=0$.} 
\item[(A1)] {\em Let $J \subsetneqq S$, $M \in \Irr(W_J)$ and
$E \in \Irr(W)$ be such that $M \uparrow E$. Then $\ba_M\leq \ba_E$.}
\item[(A2)] {\em Let $J \subsetneqq S$ and $M \in \Irr(W_J)$. Then there
exists some $E \in \Irr(W)$ such that $M \uparrow E$ and $\ba_M=\ba_E$.
In this case, we write $M \rightsquigarrow_L E$.}
\item[(A3)] {\em Let $E \in\Irr(W)$. Then there exists some $J 
\subsetneqq S$ and some $M \in \Irr(W_J)$ such that $M 
\rightsquigarrow_L E$ or $M \rightsquigarrow_L E \otimes \sgn$.}
\item[(A4)] {\em Let $E\in \Irr(W)$. Then $\ba_{E \otimes \sgn}-\ba_E=
\omega_L(E)$.}
\end{itemize}
(For the original definition of $\ba_E$ in the equal parameter case, see 
Lusztig \cite{Lusztig79b}; in that article, one can also find (A1) and (A2).
Analogous definitions and arguments work for a general weight 
function $L$; see \cite[\S 3]{my02}, \cite[Chap.~20]{Lusztig03} for 
details. (A0) is clear by the definition of $\ba_E$. A version of (A3)
for ``special'' representations in the equal parameter case already 
appeared in \cite[\S 6]{Lusztig79b}; the general case follows from 
\cite[Prop.~22.3]{Lusztig03}. Note that there does not seem to be a 
notion of ``special'' representations for the general multi-parameter 
case; see \cite[Rem.~4.11]{compf4}. (A4) follows from 
\cite[Prop.~9.4.3]{gepf}.)

Following the argument in \cite[6.5.6]{gepf}, let us now prove that 
$\ba_E=\tilde{\ba}_E$ for all $E \in \Irr(W)$.
We proceed by induction on the order of $W$. If $W=\{1\}$, then 
$\Irr(W)$ only consists of $1_W$ and we have $\ba_{1_W}= 
\tilde{\ba}_{1_W}=0$; see (A0). Now assume that $W \neq \{1\}$ and that 
the assertion is already proved for all proper parabolic subgroups of 
$W$. Consequently, using (A1), we have
\begin{equation*}
\ba_E \geq \tilde{\ba}_E^\prime \qquad \mbox{for all $E \in 
\Irr(W)$}. \tag{$*$}
\end{equation*}
Now fix $E \in \Irr(W)$. Using (A3), we distinguish two cases. Assume 
first that $M \rightsquigarrow_L E$ for some $M \in \Irr(W_J)$ where 
$J \subsetneqq S$. By ($*$), we have $\ba_E \geq \tilde{\ba}_E^\prime 
\geq \ba_M$. Since $\ba_E=\ba_M$, we deduce that $\ba_E=
\tilde{\ba}_E^\prime$. Now, by ($*$) applied to $E \otimes \sgn$, 
we also have $\ba_{E \otimes \sgn} \geq \tilde{\ba}_{E \otimes 
\sgn}^\prime$ and so, using (A4), $\tilde{\ba}_{E \otimes \sgn}^\prime - 
\tilde{\ba}_E^\prime \leq \ba_{E \otimes \sgn} -\ba_E=\omega_L(E)$.
Hence, we are in the first case of Definition~\ref{def2} and so 
$\tilde{\ba}_E= \tilde{\ba}_E^\prime=\ba_E$, as required. 

Now assume that $M \rightsquigarrow_L E \otimes \sgn$ for some $M \in
\Irr(W_J)$ where $J \subsetneqq S$. Arguing as before, we have $\ba_{E 
\otimes \sgn} = \tilde{\ba}_{E \otimes \sgn}^\prime$. Using ($*$) and (A3), 
we obtain
\[ \tilde{\ba}_{E \otimes \sgn}^\prime -\tilde{\ba}_E^\prime=
\ba_{E \otimes \sgn}-\tilde{\ba}_E^\prime\geq \ba_{E \otimes\sgn}-
\ba_E=\omega_L(E).\]
If this inequality is an equality, then $\tilde{\ba}_E^\prime=\ba_E$;
furthermore, we are in the first case of Definition~\ref{def2} and so 
$\tilde{\ba}_E=\tilde{\ba}_E^\prime= \ba_E$, as required. If the above
inequality is strict, then we are in the second case of 
Definition~\ref{def2} and, using (A4), we obtain $\tilde{\ba}_E=
\tilde{\ba}_{E\otimes \sgn}^\prime- \omega_L(E)= \ba_{E \otimes \sgn}-
\omega_L(E)=\ba_E$, as required. 
\end{rem}

\begin{rem} \label{rem4a} \rm Out of the five properties (A0)--(A4), it 
seems that (A3) is the most subtle one. In fact, (A0), (A1), (A2) and (A4) 
are proved by general arguments while the proof of (A3) relies on an
explicit case--by--case verification. Consider the following related 
statement:
\begin{itemize}
\item[(A3$^\prime$)] {\em Let $E \in \Irr(W)$ be such that $\omega_L(E) 
\geq 0$. Then there exists some proper subset $J \subsetneqq S$ and some 
$M \in \Irr(W_J)$ such that $M \uparrow E$ and $\ba_M=\ba_E$.}
\end{itemize}
Note that $\omega_L(E \otimes \sgn)=-\omega_L(E)$, so (A3$^\prime$)
certainly implies (A3). The above property has first been formulated 
and checked (in the equal parameter case) by Spaltenstein 
\cite[\S 5]{Spalt} (see also \cite[Lemma~4.9]{klord}). As far as groups 
of exceptional type are concerned, Spaltenstein just says that ``we can 
use tables''. So here is a place where CHEVIE can provide more systematic 
algorithmic verifications. It would certainly be interesting to find a 
general argument for proving (A3$^\prime$).
\end{rem}

The following definition is inspired by Lusztig \cite[4.2]{LuBook} and 
Spaltenstein \cite{Spalt}.

\begin{defn}[See \protect{\cite[2.10]{klord}}] \label{mydef} \rm We 
define a relation $\preceq_L$ on $\Irr(W)$ inductively as follows. If 
$W=\{1\}$, then $\Irr(W)$ only consists of the unit representation 
and this is related to itself. Now assume that $W \neq\{1\}$ and that 
$\preceq_L$ has already been defined for all proper parabolic subgroups 
of $W$. Let $E,E'\in\Irr(W)$. Then we write $E\preceq_L E'$ if there 
is a sequence $E=E_0,E_1,\ldots, E_m=E'$ in $\Irr(W)$ such that, for 
each $i \in \{0,1,\ldots,m-1\}$, the following condition is satisfied. 
There exists a subset $I_i \subsetneqq S$ and $M_i',M_i''\in 
\Irr(W_{I_i})$, where $M_i' \preceq_L M_i''$ within $\Irr(W_{I_i})$, such
that either
\begin{align*}
M_i' \uparrow E_{i-1} \quad &\mbox{and} \quad M_i'' \uparrow E_i 
\quad \mbox{where} \quad \tilde{\ba}_{E_i}= \tilde{\ba}_{M_i''}\\
\intertext{or} M_i' \uparrow E_i \otimes \sgn\quad & \mbox{and} \quad M_i''
\uparrow  E_{i-1}\otimes \sgn \quad \mbox{where} \quad \tilde{\ba}_{E_{i-1} 
\otimes \sgn}=\tilde{\ba}_{M_i''}.
\end{align*}
Let $\sim_L$ be the equivalence relation associated with $\preceq_L$,
that is, we have $E \sim_L E'$ if and only if $E \preceq_L E'$ and
$E' \preceq_L E$. Then we have an induced partial order on the set
of equivalence classes of $\Irr(W)$ which we denote by the same
symbol $\preceq_L$.
\end{defn}

\begin{exmp} \label{expmydef} \rm (a) If $L(s)=0$ for all $s \in S$,
then $E \preceq_L E'$ for any $E,E' \in \Irr(W)$.

(b) Assume that $W \cong \fS_n$ and $L(s)=a>0$ for $s \in S$, as in
Example~\ref{asym}. Let $\lambda,\mu$ be partitions of $n$. Then we have
$E^\lambda \preceq_L E^\mu$ if and only if $\lambda \trianglelefteq \mu$, 
where $\trianglelefteq$ denotes the {\em dominance order} on partitions;
see \cite[Exp.~3.5]{klbn}.
\end{exmp}

\begin{rem} \label{rem5} \rm One can show that the following 
``monotony'' property holds:
\begin{itemize}
\item[(a)] {\em If $E, E' \in \Irr(W)$ are such that $E \preceq_L E'$, 
then $\tilde{\ba}_{E'} \leq \tilde{\ba}_E$;}
\end{itemize}
see \cite[Prop.~4.4]{klord}, \cite[\S 6]{klbn}. Consequently, the 
equivalence classes of $\Irr(W)$ under $\sim_L$ are precisely the 
``families'' as defined by Lusztig \cite[4.2]{LuBook}, 
\cite[23.1]{Lusztig03}. (This immediately follows from the definitions, 
see the argument in \cite[Prop.~4.4]{klord}.)  In particular, the 
following holds:
\begin{itemize}
\item[(b)] {\em The function $E \mapsto \tilde{\ba}_E$ is constant on 
the ``families'' of $\Irr(W)$.}
\end{itemize}
In the equal parameter case, this appeared originally in 
\cite[4.14.1]{LuBook}; see also \cite{Lusztig82}.
\end{rem}

It is straightforward to implement the recursion in Definition~\ref{mydef} 
in the GAP programming language. In this way, one can for example 
systematically re-compute the families of $\Irr(W)$ for the exceptional 
types (in the equal parameter case), which are listed in 
\cite[Chap.~4]{LuBook}. Similar computations can be performed for a 
general weight function $L$. 

\begin{exmp} \label{monoF4} \rm Let $W$ be of type $F_4$ with generators
labelled as follows:
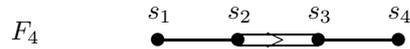
\begin{center}
\begin{picture}(200,20)
\put( 10, 5){$F_4$}
\put( 61,13){$s_1$}
\put( 91,13){$s_2$}
\put(121,13){$s_3$}
\put(151,13){$s_4$}
\put( 65, 5){\circle*{5}}
\put( 95, 5){\circle*{5}}
\put(125, 5){\circle*{5}}
\put(155, 5){\circle*{5}}
\put(105,2.5){$>$}
\put( 65, 5){\line(1,0){30}}
\put( 95, 7){\line(1,0){30}}
\put( 95, 3){\line(1,0){30}}
\put(125, 5){\line(1,0){30}}
\end{picture}
\end{center}
Assume that $a:=L(s_1)=L(s_2)>0$ and $b:=L(s_3)=L(s_4)>0$. By the 
symmetry of the Dynkin diagram, we may also assume without loss of 
generality that $b\geq a$.  The results of the computation of 
$\preceq_L$ and $\sim_L$ are presented in Table~\ref{lrgraphF4}. The 
notation for $\Irr(W)$ follows \cite[App.~C]{gepf}; for example, $1_1=
1_W$, $1_4=\sgn$ and $4_2$ is the standard reflection representation. 

Quite remarkably, it turns out that there are only $4$ essentially 
different cases. Note that, a priori, one has to deal with infinitely 
many values of $a,b$; a reduction to a finite set of values is achieved 
by using similar techniques as in \cite{compf4}; in any case, the final 
result is the same as that given in the table in \cite[p.~362]{compf4}. 

The partition of $\Irr(W)$ into {\em families} follows from the 
earlier results of Lusztig \cite[22.17]{Lusztig03}. (Note that there is 
an error for $b=2a$ in \cite[22.17]{Lusztig03}; this has been corrected 
in \cite[4.10]{compf4}, based on the explicit computations using CHEVIE.)

This example, and Guilhot's results \cite{guil} on affine Weyl groups of
rank~$2$ (which also rely on explicit computations using GAP), provide  
considerable evidence in support of Bonnaf\'e's ``{\em semicontinuity 
conjectures}'' \cite{semic}.
\end{exmp}

\begin{table}[htbp] \caption{Partial order $\preceq_L$ on families 
in type $F_4$} \label{lrgraphF4} 
{\begin{center}
\renewcommand{\arraystretch}{1.3} \begin{tabular}{l}
\begin{picture}(325,265)
\put( 16,  0){$a=b$}
\put( 27, 20){\circle{4}}
\put( 32, 18){\scriptsize{$1_4$}}
\put( 27, 22){\line(0,1){16}}
\put( 27, 40){\circle{4}}
\put( 32, 38){\scriptsize{\fbox{$4_5$}}}
\put( 27, 42){\line(0,1){16}}
\put( 27, 60){\circle{4}}
\put( 32, 58){\scriptsize{$9_4$}}
\put( 28, 62){\line(4,5){13}}
\put( 26, 62){\line(-4,5){13}}
\put( 12, 80){\circle{4}}
\put(  0, 78){\scriptsize{$8_2$}}
\put( 42, 80){\circle{4}}
\put( 47, 78){\scriptsize{$8_4$}}
\put( 13, 82){\line(4,5){13}}
\put( 41, 82){\line(-4,5){13}}
\put( 27,100){\circle{4}}
\put(  0, 99){\scriptsize{\fbox{$12_1$}}}
\put( 28,102){\line(4,5){13}}
\put( 26,102){\line(-4,5){13}}
\put( 12,120){\circle{4}}
\put(  0,118){\scriptsize{$8_1$}}
\put( 42,120){\circle{4}}
\put( 47,118){\scriptsize{$8_3$}}
\put( 13,122){\line(4,5){13}}
\put( 41,122){\line(-4,5){13}}
\put( 27,140){\circle{4}}
\put( 32,140){\scriptsize{$9_1$}}
\put( 27,142){\line(0,1){16}}
\put( 27,160){\circle{4}}
\put( 32,158){\scriptsize{\fbox{$4_2$}}}
\put( 27,162){\line(0,1){16}}
\put( 27,180){\circle{4}}
\put( 32,178){\scriptsize{$1_1$}}
\put( 93,  0){$b=2a$}
\put(105, 20){\circle{4}}
\put(110, 18){\scriptsize{$1_4$}}
\put(105, 22){\line(0,1){16}}
\put(105, 40){\circle{4}}
\put(110, 38){\scriptsize{$2_4$}}
\put(105, 42){\line(0,1){16}}
\put(105, 60){\circle{4}}
\put(110, 58){\scriptsize{$4_5$}}
\put(105, 62){\line(0,1){16}}
\put(105, 80){\circle{4}}
\put( 84, 78){\scriptsize{\fbox{$1_2$}}}
\put(106, 82){\line(1,1){14}}
\put(104, 82){\line(-1,2){11.7}}
\put(121, 98){\circle{4}}
\put(126, 96){\scriptsize{$4_3$}}
\put(121,100){\line(0,1){14}}
\put(121,116){\circle{4}}
\put(126,114){\scriptsize{$9_2$}}
\put( 91,107){\circle{4}}
\put( 79,105){\scriptsize{$8_2$}}
\put(106,132){\line(1,-1){14}}
\put(104,132){\line(-1,-2){11.7}}
\put(105,134){\circle{4}}
\put( 80,132){\scriptsize{\fbox{$16_1$}}}
\put(106,136){\line(1,1){14}}
\put(104,136){\line(-1,2){11.7}}
\put(121,152){\circle{4}}
\put(126,150){\scriptsize{$9_3$}}
\put(121,154){\line(0,1){14}}
\put(121,170){\circle{4}}
\put(126,168){\scriptsize{$4_4$}}
\put( 91,161){\circle{4}}
\put( 79,159){\scriptsize{$8_1$}}
\put(106,186){\line(1,-1){14}}
\put(104,186){\line(-1,-2){11.7}}
\put(105,188){\circle{4}}
\put( 84,186){\scriptsize{\fbox{$1_3$}}}
\put(105,190){\line(0,1){16}}
\put(105,208){\circle{4}}
\put(110,206){\scriptsize{$4_2$}}
\put(105,210){\line(0,1){16}}
\put(105,228){\circle{4}}
\put(110,226){\scriptsize{$2_3$}}
\put(105,230){\line(0,1){16}}
\put(105,248){\circle{4}}
\put(110,246){\scriptsize{$1_1$}}
\put(174,  0){$2a>b>a$}
\put(195, 20){\circle{4}}
\put(200, 18){\scriptsize{$1_4$}}
\put(195, 22){\line(0,1){10}}
\put(195, 34){\circle{4}}
\put(200, 32){\scriptsize{$2_4$}}
\put(195, 36){\line(0,1){10}}
\put(195, 48){\circle{4}}
\put(200, 46){\scriptsize{$4_5$}}
\put(195, 50){\line(0,1){10}}
\put(195, 62){\circle{4}}
\put(200, 60){\scriptsize{$2_2$}}
\put(195, 64){\line(0,1){10}}
\put(195, 76){\circle{4}}
\put(181, 74){\scriptsize{$9_4$}}
\put(197, 77){\line(3,2){15.7}}
\put(215, 88){\circle{4}}
\put(220, 86){\scriptsize{$8_4$}}
\put(215, 90){\line(0,1){10}}
\put(215,102){\circle{4}}
\put(220,100){\scriptsize{$1_2$}}
\put(215,104){\line(0,1){10}}
\put(215,116){\circle{4}}
\put(220,114){\scriptsize{$4_3$}}
\put(215,118){\line(0,1){10}}
\put(215,130){\circle{4}}
\put(220,128){\scriptsize{$9_2$}}
\put(195,142){\circle{4}}
\put(168,140){\scriptsize{\fbox{$16_1$}}}
\put(197,141){\line(3,-2){15.7}}
\put(180,109){\circle{4}}
\put(167,107){\scriptsize{$8_2$}}
\put(180,111){\line(1,2){14.5}}
\put(180,107){\line(1,-2){14.5}}
\put(197,143){\line(3,2){15.7}}
\put(215,154){\circle{4}}
\put(220,152){\scriptsize{$9_3$}}
\put(215,156){\line(0,1){10}}
\put(215,168){\circle{4}}
\put(220,166){\scriptsize{$4_4$}}
\put(215,170){\line(0,1){10}}
\put(215,182){\circle{4}}
\put(220,180){\scriptsize{$1_3$}}
\put(215,184){\line(0,1){10}}
\put(215,196){\circle{4}}
\put(220,194){\scriptsize{$8_3$}}
\put(195,208){\circle{4}}
\put(183,206){\scriptsize{$9_1$}}
\put(197,207){\line(3,-2){15.7}}
\put(180,175){\circle{4}}
\put(168,173){\scriptsize{$8_1$}}
\put(180,177){\line(1,2){14.5}}
\put(180,173){\line(1,-2){14.5}}
\put(195,210){\line(0,1){10}}
\put(195,222){\circle{4}}
\put(200,220){\scriptsize{$2_1$}}
\put(195,224){\line(0,1){10}}
\put(195,236){\circle{4}}
\put(200,234){\scriptsize{$4_2$}}
\put(195,238){\line(0,1){10}}
\put(195,250){\circle{4}}
\put(200,248){\scriptsize{$2_3$}}
\put(195,252){\line(0,1){10}}
\put(195,264){\circle{4}}
\put(200,262){\scriptsize{$1_1$}}
\put(284,  0){$b>2a$}
\put(295, 20){\circle{4}}
\put(300, 18){\scriptsize{$1_4$}}
\put(295, 22){\line(0,1){16}}
\put(295, 40){\circle{4}}
\put(300, 38){\scriptsize{$2_4$}}
\put(296, 42){\line(1,1){15.5}}
\put(294, 42){\line(-1,1){15.5}}
\put(277, 59){\circle{4}}
\put(264, 57){\scriptsize{$1_2$}}
\put(313, 59){\circle{4}}
\put(318, 57){\scriptsize{$4_5$}}
\put(296, 76){\line(1,-1){15.5}}
\put(294, 76){\line(-1,-1){15.5}}
\put(295, 78){\circle{4}}
\put(301, 75){\scriptsize{$8_4$}}
\put(296, 80){\line(3,2){15.5}}
\put(294, 80){\line(-1,1){17}}
\put(313, 92){\circle{4}}
\put(318, 90){\scriptsize{$9_4$}}
\put(311, 93){\line(-5,4){33.5}}
\put(313, 94){\line(0,1){14}}
\put(313,110){\circle{4}}
\put(318,108){\scriptsize{$2_2$}}
\put(313,112){\line(0,1){14}}
\put(313,128){\circle{4}}
\put(318,126){\scriptsize{$8_2$}}
\put(296,140){\line(3,-2){15.5}}
\put(295,142){\circle{4}}
\put(266,140){\scriptsize{\fbox{$16_1$}}}
\put(276, 99){\circle{4}}
\put(264, 97){\scriptsize{$4_3$}}
\put(276,101){\line(0,1){18}}
\put(276,121){\circle{4}}
\put(264,119){\scriptsize{$9_2$}}
\put(294,140){\line(-1,-1){17}}
\put(296,144){\line(3,2){15.5}}
\put(294,144){\line(-1,1){17}}
\put(313,156){\circle{4}}
\put(318,154){\scriptsize{$8_1$}}
\put(313,158){\line(0,1){14}}
\put(313,174){\circle{4}}
\put(318,172){\scriptsize{$2_1$}}
\put(313,176){\line(0,1){14}}
\put(313,192){\circle{4}}
\put(318,190){\scriptsize{$9_1$}}
\put(311,191){\line(-5,-4){33.5}}
\put(296,204){\line(3,-2){15.5}}
\put(295,206){\circle{4}}
\put(301,204){\scriptsize{$8_3$}}
\put(276,163){\circle{4}}
\put(264,161){\scriptsize{$9_3$}}
\put(276,165){\line(0,1){18}}
\put(276,185){\circle{4}}
\put(264,183){\scriptsize{$4_4$}}
\put(294,204){\line(-1,-1){17}}
\put(296,208){\line(1,1){15.5}}
\put(294,208){\line(-1,1){15.5}}
\put(277,225){\circle{4}}
\put(264,223){\scriptsize{$1_3$}}
\put(313,225){\circle{4}}
\put(318,223){\scriptsize{$4_2$}}
\put(296,242){\line(1,-1){15.5}}
\put(294,242){\line(-1,-1){15.5}}
\put(295,244){\circle{4}}
\put(300,242){\scriptsize{$2_3$}}
\put(295,246){\line(0,1){16}}
\put(295,264){\circle{4}}
\put(300,262){\scriptsize{$1_1$}}
\end{picture} \\
{\small A box indicates a family containing several irreducible 
representations:}\\
{\small $\mbox{\fbox{$4_2$}}=\{2_1,2_3,4_2\},\;\;
\mbox{\fbox{$4_5$}}=\{2_2,2_4,4_5\},\;\;
\mbox{\fbox{$1_3$}}=\{1_3,2_1,8_3,9_1\},\;\;
\mbox{\fbox{$1_2$}}=\{1_2,2_2,8_4,9_4\}$,} \\
{\small $\qquad\mbox{\fbox{$12_1$}}=\{1_2,1_3,4_1,4_3,4_4,6_1,6_2,
9_2,9_3,12_1, 16_1\},\quad \mbox{\fbox{$16_1$}}=\{4_1,6_1,6_2,12_1,16_1\}$.}
\\ {\small Otherwise, the family contains just one irreducible 
respresentation.}
\end{tabular}
\end{center}}
\end{table}

\begin{rem} \label{rem6} \rm The idea of partitioning $\Irr(W)$ into 
``families'' originally arose from the representation theory of finite 
groups of Lie type, see Lusztig \cite[\S 8]{Lusztig79a}. A completely new 
interpretation appeared in the theory of {\em Kazhdan--Lusztig cells}; 
see \cite{KaLu}, \cite{Lu83}. Among others, this gives rise not only to a
partition but to a natural pre-order relation $\leq_{\cLR}$ on $\Irr(W)$; 
see \cite[5.15]{LuBook}, \cite[Def.~2.2]{klord}. The relation $\leq_{\cLR}$ 
is an essential ingredient, for example, in the construction of a 
``cellular structure'' in the generic Iwahori--Hecke algebra associated 
with $W,L$; see \cite{mycell}, \cite{myedin}. One can show by a general 
argument that 
\[ E \preceq_L E' \qquad \Rightarrow \qquad E \leq_{\cLR} E'
\qquad\quad (E,E' \in \Irr(W));\]
see \cite[Prop.~3.4]{klord}. In the equal parameter case, it is known
that the reverse implication also holds; see \cite[Theorem~4.11]{klord}.
The computations involved in Example~\ref{monoF4} provide considerable 
evidence that this will also hold for general weight functions 
$L$.---Thus, $\preceq_L$ may be regarded as a purely combinatorial 
(and computable!) characterisation of $\leq_{\cLR}$. 
\end{rem}

Finally, let us assume that $W$ is the Weyl group of a connected 
reductive algebraic group $G$ over $\overline{\F}_p$ where $p$ is
a good prime. Let $\cN_G$ be the set of all pairs $(\cO,\cL)$ where 
$\cO$ is a unipotent class in $G$ and $\cL$ is a $G$-equivariant 
irreducible $\overline{\Q}_\ell$-local system on $\cO$ (up to 
isomorphism); here, $\ell$ is a prime different from $p$. By the Springer 
correspondence (see \cite{Spr}, \cite{LuIC}), we obtain a natural 
injective map 
\[ \Irr(W)\;\hookrightarrow\;\cN_G,\qquad E\;\mapsto \;\iota_E=(\cO_E,
\cL_E).\]
It is known that, for any unipotent class $\cO$, the pair $(\cO,
\overline{\Q}_\ell)\in \cN_G$ (where $\overline{\Q}_\ell$ stands for the 
trivial local system) is in the image of this map. Hence, the map
\[ \Irr(W) \;\rightarrow\; \{\mbox{unipotent classes of $G$}\},\qquad
E \; \mapsto\; \cO_E,\]
is surjective. 

\begin{rem} \label{remspr} \rm The Springer correspondence is explicitly
known in all cases. In good characteristic, the results are systematically
presented in Section~13.3 of Carter \cite{Carter2}; for bad characteristic, 
see \cite{LuSp}, \cite{Spa2}. It turns out that $\cN_G$ and the
map $E \mapsto \iota_E$ are independent of $p$ (in a suitable sense)
as long as $p$ is good; some compatibility properties of the Springer
correspondence in good and bad characteristic are established in 
\cite[\S 2]{GeMa2}. 
\end{rem}

\begin{rem} \label{defsp1} \rm Let $\Gamma=\Z$ and consider the ``equal 
parameter'' weight function $L_0$ such that $L_0(s)=1$ for all $s \in S$.  
Let $\cF\subseteq \Irr(W)$ be a family with respect to $L_0$ (see 
Remark~\ref{rem5}) and consider the following collection of unipotent 
classes in $G$:
\[\cC(\cF):=\{\cO_E\mid E\in\Irr(W) \mbox{ such that } E \in \cF\}.\]
Then it is known that there exists a unique unipotent class in
$\cC(\cF)$, which we denote by $\cO_\cF$, such that $\cO\subseteq
\overline{\cO}_{\cF}$ for all $\cO\in\cC(\cF)$; see 
\cite[Prop.~2.2]{GeMa2}. (Here, and below, $\overline{X}$ denotes the
Zariski closure in $G$ for any subset $X \subseteq G$.) Thus, $\cO_{\cF}$
is the maximum of the elements in $\cC(\cF)$ with respect to the partial
order given by the Zariski closure. A unipotent class of the form 
$\cO_\cF$ will be called a ``{\em special}'' unipotent class. Thus, we 
have a bijection
\[ \{\mbox{families of $\Irr(W)$}\} \; \stackrel{1{-}1}{\longrightarrow}\;
\{\mbox{special unipotent classes of $G$}\},\quad\cF\mapsto\cO_\cF.\]
(Special unipotent classes were originally defined by Lusztig 
\cite[\S 9]{Lusztig79b}. The above equivalent characterisation appeared in 
\cite[5.2]{klord}; see also Remark~\ref{defsp2} below.)
\end{rem}

Now we can formulate the following geometric interpretation of
the pre-order relation $\preceq_{L_0}$ in Definition~\ref{mydef}.

\begin{thm}[Spaltenstein \protect{\cite{Spalt}}] \label{thmspalt}
Let $\cF, \cF'$ be families in $\Irr(W)$ (with respect to the equal
parameter weight function $L_0$). Then we have 
\[ \cF \preceq_{L_0} \cF' \qquad \Leftrightarrow \qquad \cO_{\cF}
\subseteq \overline{\cO}_{\cF'}.\]
\end{thm}

Spaltenstein uses a slightly different definition of $\preceq_{L_0}$; the 
equivalence with the one in Definition~\ref{mydef} is shown in 
\cite[Cor.~5.6]{klord}. The proofs rely on some explicit verifications 
for exceptional types; Spaltenstein just says that ``we can then use 
tables'' \cite[p.~215]{Spalt}. So here again, CHEVIE provides a more 
systematic algorithmic way of verifying such statements. 

\begin{rem} \label{defsp2} \rm Let $\cS_W$ be the set of all $E \in 
\Irr(W)$ such that $\cL_E\cong \overline{\Q}_\ell$ and $\cO_E$ is a 
special unipotent class; see Remark~\ref{defsp1}. Then every family 
of $\Irr(W)$ as above contains a unique representation in $\cS_W$. It 
is known that $\cS_W$ is the set of ``special'' representations of $W$
as defined by Lusztig \cite{Lusztig79a}, \cite{Lusztig79b}. (This follows 
from \cite[Prop.~2.2]{GeMa2}.) Following Lusztig \cite{Lspec}, we define 
the ``{\em special piece}'' corresponding to $E \in \cS_W$ to be the 
set of all elements in $\overline{\cO}_E$ which are not contained in 
$\overline{\cO}_{E'}$ where $E' \in \cS_W$ is such that $\cO_{E'}
\subsetneqq \overline{\cO}_E$. By Spaltenstein \cite{Spa1} and Lusztig 
\cite{Lspec}, the various special pieces form a partition of 
$G_{\text{uni}}$. Note that every special piece is a union of a special 
unipotent class (which is open dense in the special piece) and of a 
certain number (possibly zero) of non-special unipotent classes.---We
will encounter the special pieces of $G_{\text{uni}}$ again
in Conjecture~\ref{conj53} below.
\end{rem}

\section{Green functions} \label{secgreen}
We begin by describing a basic algorithm which is inspired by the 
computation of Green functions and \cite{GeMa1}. It can be formulated 
without any reference to algebraic groups; in fact, it will work for 
any finite Coxeter group $W$ (including the dihedral groups and groups 
of type $H_3$, $H_4$). Let $u$ be an indeterminate over $\Q$. We define 
a matrix
\[\Omega=\bigl(\omega_{E,E'}\bigr)_{E,E' \in \Irr(W)},\]
as follows. Let $D_W:=u^{l(w_0)}(u-1)^{|S|} \sum_{w \in W} u^{l(w)}$ 
where $w_0 \in W$ is the longest element. Then, for any $E,E' \in \Irr(W)$, 
we set 
\[ \omega_{E,E'}:=\frac{D_W}{|W|}\sum_{w \in W} \frac{\mbox{trace}(w,E)\,
\mbox{trace}(w,E')}{\det(u\, \mbox{id}_V -w)}\, \in \Q(u);\]
here, $W$ is regarded as a subgroup of $\mbox{GL}(V)$ via the natural
reflection representation on a vector space $V$ of dimension $|S|$.
It is known that $\omega_{E,E'} \in \Z[u]$ for all $E,E' \in \Irr(W)$; 
see \cite[11.1.1]{Carter2}. 

\begin{lem}[Cf.\ \protect{\cite[\S 2]{GeMa1}}] \label{lem51} Let us fix
a partition $\Irr(W)=\cI_1 \sqcup \cI_2 \sqcup  \ldots \sqcup \cI_r$ and 
a sequence of integers $b_1 \geq b_2\geq \ldots \geq b_r$. Correspondingly,
we write $\Omega$ in block form:
\[\Omega = \left[\begin{array}{@{\hspace{1mm}}c@{\hspace{1mm}}
c@{\hspace{1mm}}c@{\hspace{1mm}}c@{\hspace{1mm}}} \Omega_{1,1} & 
\Omega_{1,2} & \cdots & \Omega_{1,r} \\ \Omega_{2,1} & & & \vdots \\ 
\vdots & & & \Omega_{r{-}1,r} \\ \Omega_{r,1} & \cdots & 
\Omega_{r,r{-}1}& \Omega_{r,r} \end{array}\right]\]
where $\Omega_{i,j}$ has entries $\omega_{E,E'}$ for $E \in \cI_i$ and
$E' \in \cI_j$. Then there is a unique factorisation 
\[ \Omega=P^{\operatorname{tr}} \cdot \Lambda \cdot P, \qquad 
P=\bigl(p_{E,E'}\bigr)_{E,E' \in \Irr(W)}, \qquad \Lambda=
\bigl(\lambda_{E,E'}\bigr)_{E,E' \in \Irr(W)},\]
such that $P$ and $\Lambda$ have corresponding block shapes as follows:
\[ P=\left[\begin{array}{c@{\hspace{1mm}}c@{\hspace{1mm}}c@{\hspace{1mm}}
c@{\hspace{1mm}}} u^{b_1}I_{n_1} & P_{1,2} & \cdots & P_{1,r} \\ 0 & u^{b_2}
I_{n_2} &  & \vdots \\ \vdots & & \ddots & P_{r{-}1,r} \\ 0  & \cdots & 
0 & u^{b_r}I_{n_r} \end{array}\right] \qquad \mbox{and} 
\qquad \Lambda=\left[\begin{array}{c@{\hspace{1mm}}c@{\hspace{1mm}}
c@{\hspace{1mm}} c@{\hspace{1mm}}} \Lambda_{1} & 0 & \cdots & 0 \\ 0 &
\Lambda_{2} & & \vdots \\ \vdots & & \ddots & 0 \\ 0 & \cdots & 0 &
\Lambda_{r} \end{array}\right]; \]
here, $n_i=|\cI_i|$ and $I_{n_i}$ denotes the identity matrix of size $n_i$. 
Furthermore, the block $P_{i,j}$ has entries $p_{E,E'}\in \Q(u)$ for 
$E \in \cI_i$ and $E' \in \cI_j$; similarly, the block $\Lambda_i$ has
entries $\lambda_{E,E'}\in \Q(u)$ for $E,E' \in \cI_i$.
\end{lem}

\begin{proof} This relies on the following remark due to
Lusztig (see \cite[Lemma~2.1]{GeMa1}):
\begin{itemize}
\item[($*$)] All the principal minors of $\Omega$ are non-zero.
\end{itemize}
Now $P$ and $\Lambda$ are constructed inductively by the following 
well-known procedure (see for example \cite[Chap.~8]{dse} and note that 
$\Omega$ is symmetric). We begin with the first block column. We have 
$u^{2b_1}\Lambda_1=\Omega_{1,1}$, which determines $\Lambda_1$. For 
$i>1$ we have $u^{b_1}P_{1,i}^{\text{tr}} \Lambda_1=\Omega_{i,1}$. By
($*$), we know that $\det \Omega_{1,1}\neq 0$. Hence $\Lambda_1$ is 
invertible, and we can determine $P_{1,i}$. Now consider the $j$-th 
block column, where $j>1$. Assume that the first $j-1$ block columns
of $P$ and the first $j-1$ diagonal blocks of $\Lambda$ have already 
been determined. We have an equation
\[u^{2b_j}\Lambda_j+P_{j-1,j}^{\text{tr}}\Lambda_{j-1} P_{j-1,j}+\cdots
  + P_{1,j}^{\text{tr}} \Lambda_{1}P_{1,j}= \Omega_{j,j},\]
which can be solved uniquely for $\Lambda_j$. In particular, we have now 
determined all coefficients in $P$ and $\Lambda$ which belong to the 
first~$j$ blocks. We consider the subsystem of equations made up of these
blocks; this subsystem looks like the original system written in matrix form
above, with~$r$ replaced by~$j$. By ($*$), the right hand side has a 
non-zero determinant. Hence so have the blocks $\Lambda_1,\ldots,\Lambda_j$.
Now we can determine the coefficients of~$P$ in the $i$-th row: for $i>j$, 
we have
\[u^{b_j}P_{j,i}^{\text{tr}}\Lambda_j+P_{j-1,i}^{\text{tr}}\Lambda_{j-1} 
P_{j-1,j}+ \cdots + P_{1,i}^{\text{tr}}\Lambda_{1} P_{1,j}=\Omega_{i,j}.\]
Since $\Lambda_j$ is invertible, $P_{j,i}$ is determined. Continuing in this
way, the above system of equations is solved.
\end{proof}

\begin{exmp} \label{exp52} \rm Let $W$ be of type $B_2$, with generators
$S=\{s,t\}$. We write $\Irr(W)=\{\sgn,\sgn_2,\sgn_1,\sigma,1_W\}$ (and use
this ordering for the rows and columns of the matrices below). The values 
of the corresponding characters are obtained by formally setting $q=1$ in
the table in Example~\ref{exp2}. We have
\begin{gather*}
\det(u\,\mbox{id}_V-1)=(u-1)^2,\quad \det(u\,\mbox{id}_V-s)=
\det(u\,\mbox{id}_V-t)=u^2-1,\\ \det(u\,\mbox{id}_V-st)=u^2+1, \qquad
\det(u\,\mbox{id}_V-stst)=(u+1)^2;
\end{gather*}
furthermore, $D_W=u^4(u^2-1)(u^4-1)$. Using this information, we obtain:
\[ \Omega=\left[\begin{array}{ccccc} 
 u^8 &  u^6 &  u^6 &  u^7{+}u^5 &  u^4 \\ 
 u^6 &  u^8 &  u^4 &  u^7{+}u^5 &  u^6 \\
 u^6 &  u^4 &  u^8 &  u^7{+}u^5 &  u^6 \\
 u^7{+} u^5 &  u^7{+} u^5 &  u^7{+} u^5 &  u^8{+}2u^6{+}u^4 &  u^7{+}u^5 \\
 u^4 &  u^6 &  u^6 &  u^7{+}u^5 &  u^8 
\end{array}\right] \]
We shall now determine three factorisations of $\Omega$.

(a) Consider the partition $\Irr(W)=\{\sgn\} \sqcup \{\sgn_2\} \sqcup
\{\sgn_1,\sigma\} \sqcup \{1_W\}$, together with the sequence of 
integers $4$, $2$, $1$, $0$. We obtain the matrices:
\[ P{=}\left[\begin{array}{c@{\hspace{2mm}}c@{\hspace{2mm}}c@{\hspace{2mm}}
c@{\hspace{2mm}}c@{\hspace{2mm}}} u^4 & u^2 & u^2 & u^3{+}u & 1\\
0 & u^2 & 0 & u & 1 \\ 0 & 0 & u & 0 & 0 \\ 0 & 0 & 0 & u & 1 \\
0 & 0 & 0 & 0 & 1 \end{array}\right], \quad
\Lambda{=} \left[\begin{array}{c@{\hspace{1mm}}c@{\hspace{1mm}}
c@{\hspace{1mm}}c@{\hspace{1mm}}c@{\hspace{1mm}}} 1& 0 & 0 & 0 & 0 \\
0 & u^4{-}1 & 0 & 0 & 0 \\ 0 & 0 & u^6{-}u^2 & u^5{-}u &0\\ 0 & 0 & 
u^5{-}u & u^6{-}u^2 & 0 \\ 0 & 0 & 0 & 0 & u^8{-}u^6{-}u^4{+}u^2\end{array}
\right].\] 
(We will see below that this yields the Green functions of 
$\mbox{Sp}_4(\F_q)$, $q$ odd.)

(b) Consider the partition $\Irr(W)=\{\sgn\} \sqcup \{\sgn_2\} \sqcup 
\{\sgn_1\} \sqcup \{\sigma\} \sqcup \{1_W\}$, together with the sequence of 
integers $4$, $2$, $2$, $1$, $0$. We obtain the matrices:
\[ P{=}\left[\begin{array}{c@{\hspace{2mm}}c@{\hspace{2mm}}c@{\hspace{2mm}}
c@{\hspace{2mm}}c@{\hspace{2mm}}} u^4 & u^2 & u^2 & u^3{+}u & 1\\
0 & u^2 & 0 & u & 1 \\ 0 & 0 & u^2 & u & 1 \\ 0 & 0 & 0 & u & 1 \\
0 & 0 & 0 & 0 & 1 \end{array}\right], \;
\Lambda{=} \left[\begin{array}{c@{\hspace{1mm}}c@{\hspace{1mm}}
c@{\hspace{1mm}}c@{\hspace{1mm}}c@{\hspace{1mm}}} 1& 0 & 0 & 0 & 0 \\
0 & u^4{-}1 & 0 & 0 & \\ 0 & 0 & u^4{-}1 & 0 & 0 \\ 0 & 0 & 0 & u^6{-}
u^4{-}u^2{+}1 & 0 \\ 0 & 0 & 0 & 0 & u^8{-}u^6{-}u^4{+}u^2\end{array}
\right]\]
(We will see below that this yields the Green functions of 
$\mbox{Sp}_4(\F_q)$, $q$ even.)

(c) As in \cite[2.9]{GeMa1}, consider the partition $\Irr(W)=
\{\sgn\} \sqcup \{\sgn_2,\sgn_1, \sigma\} \sqcup \{1_W\}$, together with 
the sequence of integers $4$, $1$, $0$. We obtain the matrices:
\[ P{=}\left[\begin{array}{c@{\hspace{1mm}}c@{\hspace{1mm}}c@{\hspace{1mm}}
c@{\hspace{1mm}}c@{\hspace{1mm}}} u^4 & u^2 & u^2 & u^3{+}u & 1\\
0 & u & 0 & 0 & 0 \\ 0 & 0 & u & 0 & 0 \\ 0 & 0 & 0 & u & 1 \\
0 & 0 & 0 & 0 & 1 \end{array}\right], \;
\Lambda{=} \left[\begin{array}{c@{\hspace{1mm}}c@{\hspace{1mm}}
c@{\hspace{1mm}}c@{\hspace{1mm}}c@{\hspace{1mm}}} 1& 0 & 0 & 0 & 0 \\
0 & u^6{-}u^2 & 0 & u^5{-}u & 0 \\ 0 & 0 & u^6{-}u^2 & u^5{-}u & 0 \\
0 & u^5{-}u & u^5{-}u & u^6{+}u^4{-}u^2{-}1 & 0 \\
0 & 0 & 0 & 0 & u^8{-}u^6{-}u^4{+}u^2\end{array}
\right]\]
Quite remarkably, in all three cases the solutions are in $\Z[u]$.
(One easily finds partitions of $\Irr(W)$ for which this does not hold,
for example, $\Irr(W)=\{\sgn,\sgn_2\} \sqcup \{\sgn_1\} \sqcup \{\sigma\}
\sqcup \{1_W\}$.)
\end{exmp}

It is straightforward  to implement the algorithm in the proof of
Lemma~\ref{lem51} in the GAP programming language. In those cases 
where one expects that polynomial solutions exist, it is most efficient
to first specialise $u$ to a large number of integer values, then solve
the resulting systems of equations over $\Q$, and finally interpolate
to obtain polynomial solutions. (In order to avoid working with large
rational numbers, one can further reduce the specialised systems of 
equations modulo various prime numbers, then solve the resulting systems
over finite fields, and finally use ``chinese remainder'' techniques 
to recover the solutions over $\Q$; similar methods have been used
in the proof of \cite[Prop.~11.5.13]{gepf} where it was necessary to
invert certain matrices with polynomial entries.) All this works well 
for $W$ of rank up to $8$, including all exceptional types.

Although this turns the actual chronological development of things upside
down, the discussion in the previous section leads us to consider the 
partition of $\Irr(W)$ into families with respect to the ``equal 
parameter'' weight function $L_0\colon W \rightarrow \Z$ such that
$L_0(s)=1$ for all $s \in S$. The following conjecture has been found
through extensive experimentation with CHEVIE. It is verified for
all $W$ of exceptional type; the answer for $W$ of classical
type is open.

\begin{conj}[Geck--Malle \protect{\cite[\S 2]{GeMa1}}] \label{conj53} 
Consider the partition $\Irr(W)=\cF_1 \sqcup \ldots \sqcup \cF_r$ where 
$\cF_1,\ldots,\cF_r$ are the families with respect to $L_0$. Let $b_i$ 
be the constant value of the function $E \mapsto \tilde{\ba}_E$ on 
$\cF_i$; see Remark~\ref{rem5}. Assume that $b_1 \geq \ldots \geq b_r$. 
Let $P$ and $\Lambda$ be the matrices obtained by Lemma~\ref{lem51}. 
Then the following hold.
\begin{itemize}
\item[(a)] All the entries of $P$ and $\Lambda$ are polynomials in 
$\Z[u]$; furthermore, the polynomials in $P$ have non-negative coefficients. 
\item[(b)] Assume that $W$ is the Weyl group of a connected reductive
algebraic group $G$ over $k=\overline{\F}_p$, with a split $\F_q$-rational 
structure where $q$ is a power of $p$ (as in Remark~\ref{rem2}). Let $E_i 
\in \cS_W$. Then $\lambda_{E_i,E_i}(q)$ is the number of $\F_q$-rational 
points in the ``special piece'' corresponding to $E_i$; see 
Remark~\ref{defsp2}. 
\end{itemize}
\end{conj}

We now turn to the discussion of Green functions. Let $G$ be a connected 
reductive algebraic group over $k=\overline{\F}_p$. Let $B \subseteq G$ be 
a Borel subgroup and $T \subseteq G$ be a maximal torus contained in $B$. 
Let $W=N_G(T)/T$ be the Weyl group of $G$. Let $q$ be a power of $p$ and 
$F \colon G \rightarrow G$ be a Frobenius map with respect to a split 
$\F_q$-rational structure on $G$, as in Remark~\ref{rem2}. Recall that 
then $B$ and all unipotent classes of $G$ are $F$-stable; furthermore, 
$F$ acts as the identity on $W$.  

Let $w\in W$ and $T_w \subseteq G$ be an $F$-stable maximal torus
obtained from $T$ by twisting with~$w$. Let $\theta \in \Irr(T_w^F)$ and
$R_{T_w}^\theta$ be the character of the corresponding virtual 
representation of $G^F$ defined by Deligne and Lusztig; see Carter 
\cite[\S 7.2]{Carter2}. Then the restriction of $R_{T_w}^\theta$ to 
$G_{\text{uni}}^F$ only depends on $w$ but not on $\theta$ (see 
\cite[7.2.9]{Carter2}). This restriction is called 
the {\em Green function} corresponding to $w \in W$; it will be denoted 
by $Q_w$. There is a character formula which reduces the computation of 
the values of $R_{T_w}^\theta$ to the computation of the values of 
various Green functions (see \cite[7.2.8]{Carter2}). It is known that 
the values of $Q_w$ are integers (see \cite[\S 7.6]{Carter2}), but it 
is a very hard problem to compute these values explicitly.

Let $E \in \Irr(W)$. Following Lusztig \cite[\S 3.7]{LuBook},
we define 
\[ R_E:=\frac{1}{|W|} \sum_{w \in W} \mbox{trace}(w,E)\, R_{T_w}^1
\qquad \mbox{and} \qquad Q_E:=R_E|_{G_{\text{uni}}^F},\]
where the superscript $1$ stands for the unit representation of $T_w^F$.
Note that $Q_w=\sum_{E \in \Irr(W)} \mbox{trace}(w,E)\, Q_E$, 
so $Q_E$ and $Q_w$ determine each other. Now the entries of the 
matrix $\Omega$ introduced above have the following interpretation:
\[\omega_{E,E'}(q)=\sum_{u \in G_{\text{uni}}^F} Q_E(u) Q_{E'}(u)
\qquad (E,E' \in \Irr(W)).\]
(This follows from the orthogonality relations for Green functions; 
see \cite[7.6.2]{Carter2}. It also uses the formulae for $|T_w^F|$ 
and $|(N_G(T_w)/T_w)^F|$ in \cite[\S 3.3]{Carter2}.)

As in the previous section, let $\cN_G$ be the set of all pairs 
$(\cO,\cL)$ where $\cO$ is a unipotent class in $G$ and $\cL$ is a
$G$-equivariant irreducible $\overline{\Q}_\ell$-local system on 
$\cO$ (up to isomorphism). Recall that the {\em Springer 
correspondence} defines an injection
\[ \Irr(W) \hookrightarrow \cN_G, \qquad E \mapsto \iota_E=(\cO_E,
\cL_E).\]
The Frobenius map $F$ acts naturally on $\cN_G$. Given $\iota=(\cO,
\cL) \in \cN_G^F$, we obtain a class function $Y_\iota \colon G^F 
\rightarrow \C$ as in \cite[24.2.3]{L4}. We have $Y_\iota(g)=0$ unless 
$g \in \cO^F$. Furthermore, the matrix $\bigl(Y_{\iota}(g)\bigr)$ (with 
rows labelled by all $\iota=(\cO,\cL) \in \cN_G^F$ where $\cO$ is fixed, 
and columns labelled by a set of representatives of the $G^F$-classes 
contained in $\cO^F$) is, up to multiplication of the rows by roots of 
unity, the ``$F$-twisted'' character table of the finite group $A(u)=
C_G(u)/C_G(u)^\circ$ ($u \in \cO$); see \cite[24.2.4, 24.2.5]{L4}. 
In particular, the following hold:

\begin{rem} \label{rem54} \rm (a) The functions $\{Y_\iota \mid \iota 
\in \cN_G^F\}$ form a basis of the space of class functions on 
$G_{\text{uni}}^F$. 

(b) Let $\iota=(\cO,\cL)\in \cN_G^F$ where $\cL \cong \overline{\Q}_\ell$.
Then there is a root of unity $\eta$ such that $Y_\iota(g)= \eta$ for all 
$g \in \cO^F$. (It will turn out that $\eta=\pm 1$; see Remark~\ref{rem55} 
below.)
\end{rem}

Let $\cO_1,\cO_2,\ldots,\cO_r$ be the unipotent classes of $G$, where the
labelling is chosen such that $\dim \cO_1 \leq \dim \cO_2 \leq \ldots 
\leq \dim \cO_r$. For $i\in \{1,\ldots,r\}$, we set
\begin{align*}
\cI_i^*&:=\{ E \in \Irr(W) \mid \cO_E=\cO_i\}, \mbox{ and}\\
b_i^*&:=\frac{1}{2}(\dim G-\dim T- \dim \cO_i).
\end{align*}
(It is known that $\dim G -\dim T-\dim \cO_i$ always is an even number;
see \cite[5.10.2]{Carter2}.) Recall that all pairs $(\cO_i,
\overline{\Q}_\ell) \in \cN_G$ belong to the image of the Springer 
correspondence. Thus, we obtain a partition $\Irr(W)=\cI_1^* \sqcup 
\ldots \sqcup \cI_r^*$, and a decreasing sequence of integers $b_1^* 
\geq \ldots \geq b_r^*$. Hence, Lemma~\ref{lem51} yields a factorisation 
\[ \Omega=(P^*)^{\text{tr}} \cdot \Lambda^*\cdot P^*.\]
Recall that the entries of $P^*$, $\Lambda^*$ are in $\Q(u)$; we denote 
these entries by $p_{E,E'}^*$ and $\lambda_{E,E'}^*$. With this notation, 
we can now state the following fundamental result.

\begin{thm}[Springer \protect{\cite{Spr}}; Shoji \protect{\cite{S1}, 
\cite[\S 5]{S2}}; Lusztig \protect{\cite[\S 24]{L4}, \cite{L5}}; see
also \protect{\cite[\S 3]{aver}}] \label{thm3} In the above setting, 
the entries of $P^*$ and $\Lambda^*$ are polynomials in $u$. We have 
\begin{align*}
Q_E&=\sum_{E' \in \Irr(W)} p_{E',E}^*(q)\,Y_{\iota_{E'}} \quad 
\mbox{and}\\
\lambda_{E,E'}^*(q)&=\sum_{u \in G_{\operatorname{uni}}^F} 
Y_{\iota_E}(u) \, Y_{\iota_{E'}}(u)
\end{align*}
for all $E,E'\in \Irr(W)$. Furthermore, the polynomial $p_{E',E}^*$ is 
$0$ unless $\cO_{E'} \subseteq \overline{\cO}_E$. 
\end{thm}

\begin{rem} \label{rem55} \rm The above result shows that, for any 
$E \in \Irr(W)$, we have 
\[ Q_E=q^{d_E}\,Y_{\iota_E}+\sum_{\atop{E' \in \Irr(W)}{\cO_{E'} 
\subsetneqq \overline{\cO}_E}}  p_{E',E}^*(q)\,Y_{\iota_{E'}}.\]
These equations can be inverted and, hence, every function $Y_{\iota_E}$ 
can be expressed as a $\Q$-linear combination of the Green functions 
$Q_w$ ($w \in W$). Since the values of the Green functions are integers 
(see \cite[\S 7.6]{Carter2}), we deduce that the values of $Y_{\iota_E}$ 
are rational numbers. Since they are also algebraic integers, they must
be integers. In particular, the root of unity $\eta$ in 
Remark~\ref{rem54} must be $\pm 1$. 
\end{rem}

\begin{rem} \label{rem55a} \rm  (a)  Note that, in order to run the algorithm
in Lemma~\ref{lem51}, we only need to know the map $E \mapsto \cO_E$
and the dimensions $\dim \cO_E$. The finer information on the local
systems $\cL_E$ only comes in at a later stage.

(b) As formulated above, Theorem~\ref{thm3} does not say anything about 
the tricky question of determining the values of the functions 
$Y_{\iota_E}$. This relies on the careful choice of a representative 
in $\cO^F$, where the situation is optimal when a so-called ``split'' 
element can be found; see the discussion by Beynon--Spaltenstein 
\cite[\S 3]{besp}. Such split elements exist for $G$ of classical type 
in good characteristic; see Shoji \cite{S0a}. On the other hand, in 
type $E_8$ where $q \equiv -1\bmod 3$, there is one unipotent class 
which does not contain any split element; see \cite[Case~V, 
p.~591]{besp}.---For our purposes here, the information in 
Remark~\ref{rem54}(b) will be sufficient.
\end{rem}

\begin{exmp} \label{exp52a} \rm Let us re-interprete the computations 
in Example~\ref{exp52} in the light of Theorem~\ref{thm3}. By Carter 
\cite[p.~424]{Carter2} and Lusztig--Spaltenstein \cite[6.1]{LuSp}, the 
Springer correspondence for $G=\mbox{Sp}_4(k)$ is given by the following 
tables:
\[ \begin{array}{ccll} \hline \multicolumn{3}{c}{\mbox{char}(k) 
\neq 2} & b_E^*\\ \hline 
\sgn   &\mapsto & \cO_{(1111)} & \,4 \\ 
\sgn_2 &\mapsto & \cO_{(211)}  & \,2 \\
\sgn_1 &\mapsto & \cO_{(22)}   & \,1\;(\cL_E\not\cong\overline{\Q}_\ell)\\
\sigma &\mapsto & \cO_{(22)}   & \,1\\
1_W    &\mapsto & \cO_{(4)}    & \,0 \\ \hline 
\end{array}\qquad\qquad \begin{array}{cclc} \hline 
\multicolumn{3}{c}{\mbox{char}(k)=2} & b_E^* \\ \hline 
\sgn   &\mapsto & \cO_{(1111)} & 4\\ 
\sgn_2 &\mapsto & \cO_{(211)}  & 2 \\
\sgn_1 &\mapsto & \cO_{(22)}^* & 2 \\
\sigma &\mapsto & \cO_{(22)}   & 1  \\
1_W    &\mapsto & \cO_{(4)}    & 0 \\ \hline \end{array}\]
Here, $b_E^*=(\dim G-\dim T-\dim \cO_E)/2$; furthermore, $\cL_E \cong 
\overline{\Q}_\ell$ unless explicitly stated otherwise. Hence, these data 
give rise to the first two cases in Example~\ref{exp52}.
\end{exmp}

We shall now explain, following Lusztig \cite[1.2]{Lu10}, how 
the cardinalities of the sets $(\cO \cap B\dot{w}B)^F$ (see 
Section~\ref{secbruhat}) can be effectively computed. For this 
purpose, it will be convenient to introduce the following notation.

\begin{defn} \label{def3} \rm For any $E \in \Irr(W)$ and $w \in W$, 
we set 
\[ \beta_E^w:=|G^F/B^F|\sum_{1\leq i \leq d} |O_{u_i}\cap B^F
\dot{w}B^F| \, Y_{\iota_E}(u_i),\]
where $u_1,\ldots,u_d$ are representatives of the $G^F$-conjugacy 
classes contained in $\cO_E^F$. Note that, by Remark~\ref{rem54}(b), we
have 
\[ \beta_E^w=\pm |G^F/B^F||(\cO_E \cap B\dot{w}B)^F| \qquad
\mbox{if $\cL_E\cong \overline{\Q}_\ell$}.\]
\end{defn}

We will now rewrite the expression for $\beta_E^w$ using various results
from the representation theory of $G^F$. First, by Remark~\ref{rem3}(b), 
we obtain
\begin{align*}
\beta_E^w&= \sum_{1\leq i \leq d} \sum_{V  \in \Irr(\cH_q)} 
|O_{u_i}|\, \mbox{trace}(u_i,\rho_V)\, \mbox{trace}(T_w,V)\,
Y_{\iota_E}(u_i)\\
&=\sum_{V  \in \Irr(\cH_q)}\mbox{trace}(T_w,V) \sum_{u \in 
G_{\text{uni}}^F} \chi_V(u)\,Y_{\iota_E}(u),
\end{align*}
where $\chi_V$ denotes the character of $\rho_V$. Now, by definition, 
$\chi_V$ is a constituent of the character of the permutation module 
${\C}[G^F/B^F]$, and the latter is known to be equal to $R_T^1$; see 
\cite[7.2.4]{Carter2}. But then the multiplicity of $\chi_V$ in 
any $R_{T_w}^\theta$ is $0$ unless $\theta=1$; see \cite[7.3.8]{Carter2}.
Consequently, we can write 
\[ \chi_V=\Bigl(\sum_{E' \in \Irr(W)} \langle R_{E'},\chi_V\rangle \,
R_{E'}\Bigr)+\psi_V,\]
where $\langle R_{E'},\chi_V\rangle$ denotes the multiplicity of
$\chi_V$ in the decomposition of $R_{E'}$ as a linear combination of
irreducible characters; furthermore, $\psi_V$ is a class function which
is orthogonal to all $R_{T_w}^\theta$. We now use Theorem~\ref{thm3}
to evaluate $\chi_V$ on unipotent elements. Let $u \in G^F$ be unipotent.
Then  
\[ \chi_V(u)=\psi_V(u)+\sum_{E',E'' \in \Irr(W)} \langle R_{E'},\chi_V
\rangle\, p_{E''.E'}^*(q)\, Y_{\iota_{E''}}(u).\]
Consequently, we obtain
\begin{align*}
\sum_{u \in G_{\text{uni}}^F} \chi_V(u)Y_{\iota_E}(u) &=
\sum_{u \in G_{\text{uni}}^F} \psi_V(u)Y_{\iota_E}(u)\\
&\quad + \sum_{E',E'' \in \Irr(W)} \langle R_{E'},\chi_V\rangle \,
p_{E'',E'}^*(q)\, \lambda_{E'',E}^*(q).
\end{align*}
Finally, by Remark~\ref{rem55}, we can write $Y_{\iota_E}$ as a linear
combination of Green functions. Since $\psi_V$ is orthogonal to all
$R_{T_w}^\theta$, it follows that 
\[\sum_{u \in G_{\text{uni}}^F}\psi_V(u)Y_{\iota_E}(u)=0.\]
Thus, we have shown the following formula which is a slight variation
of the one obtained by Lusztig \cite[1.2(c)]{Lu10}; this is the key
to the explicit computation of $\beta_E^w$.

\begin{lem} \label{lem52} For any $E \in \Irr(W)$ and $w \in W$, we have:
\[ \beta_E^w=\sum_{V \in \Irr(\cH_q)} \sum_{E',E'' \in \Irr(W)} 
\operatorname{trace}(T_w,V)\, \langle R_{E'},\chi_V\rangle\, 
p_{E'',E'}^*(q) \, \lambda_{E'',E}^*(q).\]
\end{lem}

In the above formula, the terms ``$\mbox{trace}(T_w,V)$'' can also be seen
to be specialisations of some well-defined polynomials. For this purpose,
we introduce the generic Iwahori--Hecke algebra $\cH$ associated with $W$. 
This algebra is defined over the ring of Laurent polynomials $A={\Z}[u^{1/2},
u^{-1/2}]$; it has an $A$-basis $\{T_w\mid w \in W\}$ and the multiplication 
is given as follows, where $s \in S$ and $w \in W$:
\[T_sT_w=\left\{\begin{array}{cl}T_{sw}&\qquad\mbox{if $l(sw)>l(w)$},
\\uT_{sw}+(u-1)T_w&\qquad\mbox{if $l(sw)<l(w)$};\end{array}\right.\]
see \cite[\S 4.4]{gepf}. Thus, we have $\cH_q \cong \C \otimes_A \cH$
where $\C$ is considered as an $A$-module via the specialisation $A 
\rightarrow \C$, $u^{1/2}\mapsto q^{1/2}$; here, $q^{1/2}$ is a fixed 
square root of $q$ in $\C$. Let $K$ be the field of fractions of $A$ 
and $\cH_K$ be the $K$-algebra obtained by extending scalars from $A$
to $K$. Then it is known that $\cH_K$ is split semisimple and that there 
is a bijection $\Irr(W) \leftrightarrow \Irr(\cH_K)$, $E \leftrightarrow
E_u$, such that 
\[ \mbox{trace}(w,E)=\mbox{trace}(T_w,E_u)|_{u^{1/2}\rightarrow 1}
\qquad \mbox{for all $w \in W$};\]
see \cite[8.1.7, 9.3.5]{gepf} or \cite[Chap.~20]{Lusztig03}. Now,
following \cite[8.2.9]{gepf}, we define the {\em character table} of 
$\cH$ by 
\[ X(\cH):=\bigl(\mbox{trace}(T_{w_C},E_u)\bigr)_{E \in \Irr(W),\,
C\in \Cl(W)},\]
where $w_C \in \Cmin$ for each $C \in \Cl(W)$. (By \cite[8.2.6]{gepf},
this does not depend on the choice of the elements $w_C$; by 
\cite[9.3.5]{gepf}, the entries of $X(\cH)$ are in $\Z[u^{1/2}]$.) 

Finally, since the algebra $\cH_q$ is semisimple, we also have a bijection
$\Irr(\cH_q) \leftrightarrow \Irr(\cH_K)$, $V \leftrightarrow V_u$, 
such that 
\[ \mbox{trace}(T_w,V)=\mbox{trace}(T_w,V_u)|_{u^{1/2}\rightarrow q^{1/2}}
\qquad \mbox{for all $w \in W$}.\]
Composing this bijection with the previous bijection $E \leftrightarrow 
E_u$, we obtain a bijection $\Irr(W) \leftrightarrow \Irr(\cH_q)$, $E 
\leftrightarrow E_q$. We now define the matrix
\[\Upsilon_W:=\bigl(\langle R_E,\chi_{E_q'}\rangle \bigr)_{E,E' 
\in \Irr(W)}.\]
The entries of this matrix are explicitly described by Lusztig's 
multiplicity formula \cite[Main Theorem 4.23]{LuBook}, together with the 
information in \cite[1.5]{Lusztig80} (for types $E_7$, $E_8$) and 
\cite[12.6]{LuBook} (in all remaining cases). It turns out that 
$\Upsilon_W$ is given by certain {\em non-abelian Fourier transformations} 
associated to the various families of $\Irr(W)$; in particular, 
$\Upsilon_W$ only depends on $W$, but not on $p$ or $q$. 

Now the three matrices $\Lambda^*$, $P^*$, $\Upsilon_W$ have rows and 
columns labelled by $\Irr(W)$; furthermore, $X(\cH)$ has rows labelled 
by $\Irr(W)$ and columns labelled by $\Cl(W)$. Consequently, it makes 
sense to consider the following product
\[ \Xi^*:=\Lambda^* \cdot P^* \cdot \Upsilon_W \cdot X(\cH),\]
which is a matrix with entries in ${\Q}[u^{1/2},u^{-1/2}]$, which has
rows labelled by $\Irr(W)$ and columns labelled by $\Cl(W)$. Then 
Lemma~\ref{lem52} can be re-stated as follows.

\begin{cor} \label{cor51} Let $E \in \Irr(W)$ and $w\in \Cmin$ for some
$C \in \Cl(W)$. Then we have:
\begin{center}
\fbox{$\beta_E^w=\mbox{$(E,C)$-entry of the matrix } \Xi^*|_{u^{1/2}
\rightarrow q^{1/2}}$}
\end{center}
In particular, the numbers $\beta_E^w$ are given by ``polynomials in $q$''.
\end{cor}

\begin{rem} \label{finrem} \rm The advantage of working with $\beta_E^w$
is that then we obtain a true expression in terms of polynomials in $q$,
as above. (In the original setting of \cite[1.2]{Lu10}, one has to 
distinguish congruence classes of $q$ modulo $3$ in type $E_8$.)
\end{rem}

Following Lusztig \cite[1.2]{Lu10}, we are now in a position to write a 
computer program for computing $\beta_E^w$ and, hence, the cardinalities 
$|(\cO\cap B\dot{w}B)^F|$. Note that:
\begin{itemize}
\item The explicit knowledge of the Springer correspondence (see 
Remark~\ref{remspr}) can be turned into a GAP/CHEVIE program which,
given any $G$, determines the partition $\Irr(W)=\cI_1^* \sqcup \ldots 
\sqcup \cI_r^*$ and the numbers $b_1^*\geq\ldots\geq b_r^*$ required 
for running the algorithm in Lemma~\ref{lem51}. (L\"ubeck 
\cite{luebeck} provides an electronic library of tables of Green functions.) 
\item The character tables of $\cH$ are known for all types of $W$; 
see Chapters~10 and~11 of \cite{gepf}. For any given $W$, they are 
explicitly available in GAP through an already existing CHEVIE function. 
\item The Fourier matrices $\Upsilon_W$ are explicitly known by 
\cite{Lusztig80}, \cite{LuBook}. They are available in GAP
through Michel's \cite{jmich} development version of CHEVIE.
\end{itemize}
It then remains to combine all these various pieces (data and algorithms) 
into a GAP program for determining $\beta_E^w$. In this way, the 
verification of Theorem~\ref{thm2} for a given $G$ is reduced to a purely 
mechanical computation. 

\begin{rem} \label{lusrem} \rm Using the methods described above,
Lusztig \cite[1.2]{Lu10} has verified Theorem~\ref{thm2} for $G$ of
exceptional type; as remarked in \cite[4.8]{Lu11}, this works both in
good and in bad characteristic. The computations also yield the 
following property of the entries of the matrix $\Xi^*$. Let $C\in 
\Cl(W)$ be cuspidal; let $\cO$ be a unipotent class in $G$ and $E 
\in \Irr(W)$ be such that $\iota_E=(\cO, \overline{\Q}_\ell)$. Then 
we have:
\begin{itemize}
\item[(a)] The $(E,C)$-entry $\Xi^*_{E,C}$ is divisible by $D_W$; recall
that we defined $D_W=u^{l(w_0)}(u-1)^{|S|}\sum_{w \in W} u^{l(w)}$ (hence,
we have $D_w(q)=G_{\text{ad}}^F$).
\item[(b)] If $\cO=\cO_C$, then the constant term of the polynomial 
$\Xi^*_{E,C}/D_W$ is $1$; otherwise, the constant term is $0$.
\end{itemize}
In fact, the further results in \cite{Lu10}, \cite{Lu11} provide a 
general proof of (a), (b), assuming that $\cO=\cO_C$. See 
\cite[4.4(a)]{Lu11} for an explicit formula for $\Xi_{E,C}^*$ in this case.
\end{rem}

\begin{table}[htbp] \caption{The numbers $|G^F/B^F|^{-1}\beta_E^w$
for $G=\text{Sp}_4(\overline{\F}_p)$} \label{tab3} \begin{center}
$\renewcommand{\arraystretch}{1.2} \begin{array}{cccccc} \hline 
p \neq 2&\cO_{(1111)} &  \cO_{(211)} & \cO_{(22)}, \cL \not\cong 
\overline{\Q}_\ell & \cO_{(22)} & \cO_{(4)} \\ \hline
1 &   1 & q^2 - 1 & q^3 - q & 2q^3 - 2q^2 & q^4 - 2q^3 + q^2 \\
t &   0 & q^3 - q^2 & 0 & q^4 - 2q^3 + q^2 & q^5 - 2q^4 + q^3 \\
stst& 0 & 0 & 0 & q^6 - 2q^5 + q^4 & q^8 - 2q^7 + q^6  \\ 
s &   0 & 0 & q^4 - q^3 & q^4 - q^3 & q^5 - 2q^4 + q^3 \\
st&   0 & 0 & 0 & 0 & q^6 - 2q^5 + q^4 \\
\hline &&&&&\\ \hline
p=2 &\cO_{(1111)} & \cO_{(211)} & \cO_{(22)}^* &\cO_{(22)} & \cO_{(4)}\\
\hline
1   & 1 & q^2 - 1 & q^2 - 1 & 2q^3 - 3q^2 + 1 & q^4 - 2q^3 + q^2 \\
t   & 0 & q^3 - q^2 & 0 & q^4 - 2q^3 + q^2 & q^5 - 2q^4 + q^3 \\
stst&0 & 0 & 0 & q^6 - 2q^5 + q^4 & q^8 - 2q^7 + q^6  \\ 
s   & 0 & 0 & q^3 - q^2 & q^4 - 2q^3 + q^2 & q^5 - 2q^4 + q^3 \\
st  &0 & 0 & 0 & 0 & q^6 - 2q^5 + q^4 \\
\hline \end{array}$
\end{center}
\end{table}

We illustrate all this with our usual example $G=\mbox{Sp}_4
(\overline{\F}_p)$. Recall that we write $\Irr(W)=\{\sgn, \sgn_2,
\sgn_1,\sigma,1_W\}$. Then $\Upsilon_W$ is given by 
\[ \Upsilon=\frac{1}{2}\left[\begin{array}{crrrc} 1 & 0 & 0 & 0 & 0 \\
0 &  1 & {-}1 & 1 & 0 \\ 
0 & {-}1 &  1 & 1 & 0 \\ 
0 &  1 &  1 & 1 & 0 \\ 
0 & 0 & 0 & 0 & 1 \end{array}\right]\] 
(where the rows and columns are labelled by $\Irr(W)$ as specified above). 
All the remaining pieces of information are already contained in the
examples considered earlier; see Example~\ref{exp2} for the character 
table $X(\cH)$. The results are contained in Table~\ref{tab3}. First 
of all note that this is, of course, consistent with the computations 
in Example~\ref{exp2}. Furthermore, the entries in the rows corresponding
to the two cuspidal classes (with representatives $st$ and $stst$) are 
divisible by $|B^F|$, which implies that the properties (a) and (b) 
in Remark~\ref{lusrem} hold. 

\begin{rem} \label{finh} \em Finally, we wish to state a conjecture 
concerning a general finite Coxeter group $W$. We place ourselves in 
the setting of Conjecture~\ref{conj53} where $P,\Lambda$ are computed 
with respect to the partition of $\Irr(W)$ into Lusztig's families, using
the equal parameter weight function $L_0$. We form again the matrix 
\[ \Xi:=\Lambda \cdot P \cdot \Upsilon_W \cdot X(\cH) \qquad
\mbox{(with entries in $\Q(u^{1/2})$)}.\]
Analogues of the Fourier matrix $\Upsilon_W$ for $W$ of type $I_2(m)$, 
$H_3$ and $H_4$ have been constructed by Lusztig \cite[\S 3]{Luexo}, 
Brou\'e--Malle \cite[7.3]{brma} and Malle \cite{maexo}, respectively. 

Now let $C \in \Cl(W)$. Then we conjecture that there is a unique family
of $\Irr(W)$, denoted by $\cF_C$, with the following properties:
\begin{itemize}
\item[(a)] {\em For some $w \in \Cmin$ and some $E \in \cF$, the 
$(E,C)$-entry of $\Xi$ is non-zero.}
\item[(b)] {\em For any $w' \in \Cmin$ and any $E' \in \Irr(W)$, we have 
that the $(E',C)$-entry is zero unless $\cF_C \preceq_{L_0} \cF'$, where 
$\cF'$ is the family containing $E'$.}
\end{itemize}
(Here, $\preceq_{L_0}$ is the partial order as in Definition~\ref{mydef}.)
Furthermore, we expect that the assignment $C \mapsto \cF_C$ defines a 
surjective map from $\Cl(W)$ to the set of families of $\Irr(W)$. In
particular, we would obtain a natural partition of $\Cl(W)$ into
pieces which are indexed by the families of $\Irr(W)$; a similar idea 
has been formulated by Lusztig \cite[1.4]{Lu09} (for $W$ of 
crystallographic type).

For example, if $W \cong \fS_n$ (type $A_{n-1}$), then the above conjecture 
is equivalent to Theorem~\ref{thm2}. In this case, all families are 
singleton sets; furthermore, both $\Cl(W)$ and $\Irr(W)$ are naturally 
parametrised by the partitions of $n$. The map $C \mapsto \cF_C$ is given 
by sending the conjugacy class of $W$ consisting of elements of cycle type 
$\lambda \vdash n$ to the family consisting of the irreducible 
representation labelled by $\lambda$. 

Using the computational methods described above, one can check that the
conjecture holds for all $W$ of exceptional type. The resulting maps 
$C\mapsto\cF_C$ for types $H_3$, $H_4$ are described in Table~\ref{tab4}, 
where we use the following conventions. In the first column, a family 
$\cF$ is specified by the unique ``special'' representation in $\cF$; see 
\cite[App.~C]{gepf}. A non-cuspidal class $C\in \Cl(W)$ is specified as 
$(w)$ where $w \in \Cmin$. Representatives for cuspidal classes have 
already been described in Table~\ref{tab1}; so, here, $\# n$ refers to 
the class with number $n$ in that table.
\end{rem}

\begin{table}[htbp] \caption{The map from $\Cl(W)$ to families in types
$H_3$ and $H_4$} \label{tab4} \begin{center}
$\begin{array}{ccl} \hline H_3 & C \mbox{ such that } \cF_C=\cF\\ \hline
1_r' & e \\ 3_s  & (1), (1212), w_0 \\ 5_r' & (13) \\ 
4_r' & (23), \#9 \\ 5_r  & (12) \\ 3_s' &  \#8 \\ 1_r  & \#6 \\ 
\hline & \\ & \\ & \\ & \\ & \\ & \\ & \end{array}\quad 
\begin{array}{ccl} \hline H_4 & C \mbox{ such that } \cF_C=\cF\\ \hline
1_r' & e \\ 
4_t' & (1), (1212), (121213212132123), w_0 \\ 
9_s' & (13), (12124), \#33 \\ 
16_{r}' & (23), (121232123), \#30, \#31, \#32 \\ 
25_r' & (134) \\ 36_{rr}' & (12) \\ 
\overline{24}_s & \left\{\begin{array}{l} (124), (243), (12123), \#19, 
\#21, \#22, \\ \#23, \#24, \#25, \#26, \#27, \#28, \#29 \end{array}\right.\\ 
36_{rr} & (123) \\ 25_r  & \#18 \\ 16_{rr} & \#17 \\ 9_s &  \#15 \\ 
4_t &  \#14 \\ 1_r & \#11 \\ \hline \end{array}$
\end{center}
\end{table}

 
\end{document}